\RequirePackage{fix-cm}
\documentclass{svjour3}                     
\smartqed  
\usepackage[utf8]{inputenc} 
\usepackage[T1]{fontenc}      
\usepackage{geometry}          
\usepackage{amsmath} 
\usepackage{amsfonts}  
\usepackage{algorithm}
\usepackage{algpseudocode}
\usepackage{nicefrac}
\usepackage{amsthm}
\usepackage{graphicx} 
\usepackage{ulem}
\usepackage[hidelinks]{hyperref}
\usepackage{caption}
\usepackage{enumitem}
\usepackage{natbib}
\usepackage{caption}

\newtheorem{rk}{Remark}[section]

\newtheorem{ass}{Assumption}[section]
\newtheorem{theo}{Theorem}[section]
\newtheorem{lem}{Lemma}[section]
\newtheorem{cor}{Corollary}[section]

\newcommand{\Acal}{\mathcal{A}}
\newcommand{\Bcal}{\mathcal{B}}

\newcommand{\Dcal}{\mathcal{D}}

\newcommand{\Fcal}{\mathcal{F}}
\newcommand{\Gcal}{\mathcal{G}}

\newcommand{\Kcal}{\mathcal{K}}

\newcommand{\Ncal}{\mathcal{N}}

\newcommand{\Tcal}{\mathcal{T}}

\newcommand{\Wcal}{\mathcal{W}}
\newcommand{\Xcal}{\mathcal{X}}

\newcommand{\Rset}{\mathbb{R}}

\newcommand{\Trans}{\scriptscriptstyle\top}

\newcommand{\ri}{\operatorname{ri}}
\newcommand{\inter}{\operatorname{int}}

\DeclareMathOperator*{\minimise}{minimise}

\DeclareMathOperator*{\argmin}{argmin}

\DeclareMathOperator*{\suprem}{sup}

\titlerunning{An Alternating Trust Region Algorithm for Distributed Linearly Constrained NLPs}

\title{An Alternating Trust Region Algorithm for Distributed Linearly Constrained Nonlinear Programs}
\subtitle{Application to the AC Optimal Power Flow}
\author{Jean-Hubert Hours and Colin N. Jones}

\institute{	
			\mailname{ Jean-Hubert Hours}\ \at
               	Automatic Control Laboratory,~Ecole Polytechnique F\'ed\'erale de Lausanne,\ \email{jean-hubert.hours@epfl.ch}
           		\and Colin N. Jones\ \at Automatic Control Laboratory,~Ecole Polytechnique F\'ed\'erale de Lausanne
}

\date{Received: date / Accepted: date}

\begin{document}
\maketitle
\begin{abstract}
A novel trust region method for solving linearly constrained nonlinear programs is presented.~The proposed technique is amenable to a distributed implementation, as its salient ingredient is an alternating projected gradient sweep in place of the~Cauchy point computation.~It is proven that the algorithm yields a sequence that globally converges to a critical point.~As a result of some changes to the standard trust region method, namely a proximal regularisation of the trust region subproblem, it is shown that the local convergence rate is linear with an arbitrarily small ratio.~Thus, convergence is locally almost superlinear, under standard regularity assumptions.~The proposed method is successfully applied to compute local solutions to alternating current optimal power flow problems in transmission and distribution networks.~Moreover, the new mechanism for computing a~Cauchy point compares favourably against the standard projected search as for its activity detection properties.
\end{abstract}
\subclassname{$49$M$27$~$\cdot$~$49$M$37$~$\cdot$~$65$K$05$~$\cdot$~$65$K$10$~$\cdot$~$90$C$06$~$\cdot$\\
$90$C$26$~$\cdot$~$90$C$30$}

\thanks{The research leading to these results has received funding from the European Research Council under the European Union's Seventh Framework Programme (FP/$2007$-$2014$)/ ERC Grant Agreement n.$307608$.}
\section{Introduction}
\label{sec:intro}
\looseness-1Minimising a separable smooth nonconvex function subject to partially separable coupling equality constraints and separable constraints, appears in many engineering problems such as Distributed Nonlinear Model Predictive Control (DNMPC)\ \citep{neco2009}, power systems\ \citep{baldick1997} and wireless networking\ \citep{chiang2007}.~For such problems involving a large number of agents, which result in large-scale nonconvex Nonlinear Programs (NLP), it may be desirable to perform computations in a distributed manner, meaning that all operations are not carried out on one single node, but on multiple nodes spread over a network and that information is exchanged during the optimisation process.~Such a strategy may prove useful to reduce the computational burden in the case of extremely large-scale problems.~Moreover, autonomy of the agents may be hampered by a purely centralised algorithm.~Case in points are cooperative tracking using DNMPC\ \citep{hours2016} or the Optimal Power Flow problem (OPF) over a distribution network\ \citep{topcu2014}, into which generating entities may be plugged or unplugged.~Moreover, it has been shown in a number of studies that distributing and parallelising computations can lead to significant speed-up in solving large-scale NLPs\ \citep{zav2008}.~Splitting operations can be done on distributed memory parallel environments such as clusters\ \citep{zav2008}, or on parallel computing architectures such as Graphical Processing Units (GPU)\ \citep{fei2014}.

\looseness-1Our objective is to develop nonlinear programming methods in which most of the computations can be distributed or parallelised.~Some of the key features of a distributed optimisation strategy are the following:
\begin{enumerate}[label=(\roman*)]  
\item \textit{Shared memory}.~Vectors and matrices involved in the optimisation process are stored on different nodes.~This requirement rules out direct linear algebra methods, which require the assembly  of matrices on a central unit.
\item \textit{Concurrency}.~A high level of parallelism is obtained at every iteration.
\item \textit{Cheap exchange}.~Global communications of agents with a central node are cheap (scalars).~More costly communications (vectors) remain local between neighbouring agents.~In general, the amount of communication should be kept as low as possible.~It is already clear that globalisation strategies based on line-search do not fit with the distributed framework\ \citep{fei2014}, as these entail evaluating a `central' merit function multiple times per iteration, thus significantly increasing communications.     
\item \textit{Inexactness}.~Convergence is `robust' to inexact solutions of the subproblems, since it may be necessary to truncate the number of sub-iterations due to communication costs.  
\item \textit{Fast convergence}.~The sequence of iterates converges at a fast (at least linear) local rate.~Slow convergence generally results in a prohibitively high number of communications.
\end{enumerate}
As we are interested in applications such as DNMPC, which require solving distributed parametric NLPs with a low latency\ \citep{hours2016}, a desirable feature of our algorithm should also be
\begin{enumerate}[label=(\roman*)]
 \setcounter{enumi}{5}
\item \textit{Warm-start and activity detection}.~The algorithm detects the optimal active-set quickly and enables warm-starting. 
\end{enumerate}
Whereas a fair number of well-established algorithms exist for solving distributed convex NLPs\ \citep{bert1997}, there is, as yet, no consensus around a set of practical methods applicable to distributed nonconvex programs.~Some work\ \citep{zav2008} exists on the parallelisation of linear algebra operations involved in solving nonconvex NLPs with\ \textsc{ipopt}\ \citep{waech2006}, but the approach is limited to very specific problem structures and the globalisation phase of\ \textsc{ipopt} (filter line-search) is not suitable for fully distributed implementations (requirements (iii), (iv) and (vi) are not met).~Among existing strategies capable of addressing a broader class of distributed nonconvex programs, one can make a clear distinction between Sequential Convex Programming (SCP) approaches and augmented Lagrangian techniques.

\looseness-1An SCP method consists in iteratively solving distributed convex NLPs, which are local approximations of the original nonconvex NLP.~To date, some of the most efficient algorithms for solving distributed convex NLPs combine dual decomposition with smoothing techniques\ \citep{neco2009,quoc2013b}.~On the contrary, an augmented Lagrangian method aims at decomposing a nonconvex auxiliary problem inside an augmented Lagrangian loop\ \citep{cohen1980,ham2011,hours2014b}. While convergence guarantees can be derived in both frameworks, computational drawbacks also exist on both sides.~For instance, it is not clear how to preserve the convergence properties of~SCP schemes when every subproblem is solved to a low level of accuracy.~Hence, (iv) is not satisfied immediately.~Nevertheless, for some recent work in this direction, one may refer to\ \citep{quoc2013a}.~The convergence rate of the algorithm analysed in\ \citep{quoc2013a} is at best linear, thus not fulfilling (v).~On the contrary, the inexactness issue can be easily handled inside an augmented Lagrangian algorithm, as global and fast local convergence is guaranteed even though the subproblems are not solved to a high level of accuracy\ \citep{solodov2012,conn1991}.~However, in practice, poor initial estimates of the dual variables can drive the iterative process to infeasible points.~Moreover, it is still not clear how the primal nonconvex subproblems should be decomposed and solved efficiently in a distributed context.~The quadratic penalty term of an augmented Lagrangian does not allow for the same level of parallelism as a (convex) dual decomposition.~Thus, requirement (ii) is not completely satisfied.~To address this issue, we have recently proposed applying Proximal Alternating Linearised Minimisations (PALM)\ \citep{bolte2013} to solve the auxiliary augmented Lagrangian subproblems\ \citep{hours2014b,hours2016}.~The resulting algorithm inherits the slow convergence properties of proximal gradient methods and does not readily allow one to apply a preconditioner.~In this paper, a novel mechanism for handling the augmented Lagrangian subproblems in a more efficient manner is proposed and analysed.~The central idea is to use alternating gradient projections to compute a~Cauchy point in a trust region~Newton method\ \citep{conn2000}.

\looseness-1When looking at practical trust region methods for solving bound-constrained problems\ \citep{zav2014}, one may notice that the safeguarded Conjugate Gradient~(sCG) algorithm is well-suited to distributed implementations, as the main computational tasks are structured matrix-vector and vector-vector multiplications, which do not require the assembly of a matrix on a central node.~Moreover, the global communications involved in an~sCG algorithm are cheap.~Thus,~sCG satisfies (i), (ii) and (iii).~The implementation of~CG on distributed architectures has been extensively explored\ \citep{verschoor2012,azevedo1993,fei2014}.~Furthermore, a trust region update requires only one centralized objective evaluation per iteration.~From a computational perspective, it is thus comparable to a dual update, which requires evaluating the constraints functional and is ubiquitous in distributed optimisation algorithms.~However, computing the~Cauchy point in a trust region loop is generally done by means of a projected line-search\ \citep{zav2014} or sequential search\ \citep{conn1988}.~Whereas it is broadly admitted that the~Cauchy point computation is cheap, this operation requires a significant amount of global communications in distributed memory parallel environments, and is thus hardly amenable to such applications\ \citep{fei2014}.~This hampers the implementability of trust region methods with good convergence guarantees on distributed computing platforms, whereas many parts of the algorithm are attractive for such implementations.~The aim of this paper is to bridge the gap by proposing a novel way of computing the~Cauchy point that is more tailored to the distributed framework.~Coordinate gradient descent methods such as~PALM, are known to be parallelisable for some partial separability structures\ \citep{bert1997}.~Moreover, in practice, the number of backtracking iterations necessary to select a block step-size, can be easily bounded, making the approach suitable for `Same Instruction Multiple Data' architectures.~Therefore, we propose using one sweep of block-coordinate gradient descent to compute a~Cauchy point.~As shown in paragraph\ \ref{subsec:activ} of Section\ \ref{sec:cv_ana}, such a strategy turns out to be efficient at identifying the optimal active-set.~It can then be accelerated by means of an inexact~Newton method.~As our algorithm differs from the usual trust region~Newton method, we provide a detailed convergence analysis in Section\ \ref{sec:cv_ana}.~Finally, one should mention a recent paper\ \citep{qi2014}, in which a trust region method is combined with alternating minimisations, namely the Alternating Directions Method of Multipliers (ADMM)\ \citep{bert1997}, but in a very different way from the strategy described next.
The contributions of the paper are the following:
\begin{itemize}
\item We propose a novel way of computing a~Cauchy point in a trust region framework, which is suitable for distributed implementations. 
\item We adapt the standard trust region algorithm to the proposed~Cauchy point computation.~Global convergence along with an almost~Q-superlinear local rate is proven under standard assumptions.
\item The proposed trust region algorithm, entitled\ \textsc{trap} (Trust Region with Alternating Projections), is used as a primal solver in an augmented Lagrangian dual loop, resulting in an algorithm that meets requirements (i)-(vi), and is applied to solve~OPF programs in a distributed fashion.
\end{itemize} 
\looseness-1In Section\ \ref{sec:back}, some basic notion in variational analysis is recalled.~In Section\ \ref{sec:algo}, our\ \textsc{trap} algorithm is presented.~Its convergence properties are analysed in Section\ \ref{sec:cv_ana}.~Global convergence to a critical point is guaranteed, as well as almost Q-superlinear local convergence.~Finally, the proposed algorithm is tested on OPF problems over transmission and distribution networks in Section\ \ref{sec:ac_pow}. 
\section{Background}
\label{sec:back}
Given a closed convex set~$\Omega$, the projection operator onto~$\Omega$ is denoted by~$P_\Omega$ and the indicator function of~$\Omega$ is defined by 
\begin{align}
\iota_{\Omega}\left(x\right)=\begin{cases}0\enspace, &\mbox{if } x\in\Omega\enspace,\\
+\infty\enspace,&\mbox{if }x\notin\Omega\enspace.\end{cases}\nonumber
\end{align}
We define the normal cone to~$\Omega$ at~$x\in\Omega$ as 
\begin{align}
\nonumber
\Ncal_\Omega\left(x\right):=\left\{v\in\Rset^d~:~\forall y\in\Omega,\left<v,y-x\right>\leq0\right\}\enspace. 
\end{align}
The tangent cone to~$\Omega$ at~$x$ is defined as the closure of feasible directions at~$x$\ \citep{rock2009}.~Both~$\Ncal_\Omega\left(x\right)$ and~$\Tcal_\Omega\left(x\right)$ are closed convex cones.~As~$\Omega$ is convex, for all~$x\in\Omega$,~$\Ncal_\Omega\left(x\right)$ and~$\Tcal_\Omega\left(x\right)$ are polar to each other~\citep{rock2009}.
\begin{theo}[Moreau's decomposition\ \citep{mor1962}]
Let~$\Kcal$ be a closed convex cone in~$\Rset^d$ and~$\Kcal^{\circ}$ its polar cone.~For all~$x,y,z\in\Rset^d$, the following two statements are equivalent:
\begin{enumerate}
\item $z=x+y$ with $x\in\Kcal$,~$y\in\Kcal^{\circ}$ and $\left<x,y\right>=0$.
\item $x=P_{\Kcal}\left(z\right)$ and~$y=P_{\Kcal^{\circ}}\left(z\right)$.
\end{enumerate}
\end{theo}
The box-shaped set~$\left\{x\in\Rset^d~:~\forall{}i\in{}\left\{1,\ldots,d\right\},l_i\leq{}x_i\leq{}u_i\right\}$ is denoted by~$\mathbb{B}\left(l,u\right)$.~For $x\in\Rset^d$ and~$r>0$, the open ball of radius~$r$ centered around~$x$ is denoted by $\Bcal\left(x,r\right)$.~Given $x\in\Omega$, the set of active constraints at $x$ is denoted by~$\Acal_\Omega\left(x\right)$.~Given a set~$S\subseteq\Rset^d$, its relative interior is defined as the interior of~$S$ within its affine hull, and is denoted by~$\ri\left(S\right)$.

A critical point~$x^{\ast}$ of the function~$f+\iota_\Omega$ with~$f$ differentiable, is said to be non-degenerate if
\begin{align} 
-\nabla{}f\left(x^{\ast}\right)\in\ri\left(\Ncal_\Omega\left(x^{\ast}\right)\right)\enspace.\nonumber
\end{align}
Given a differentiable function~$f$ of several variables~$x_1,\ldots,x_N$, its gradient with respect to variable~$x_i$ is denoted by~$\nabla_if$.~Given a matrix~$M\in\Rset^{m\times n}$, its~$\left(i,j\right)$ element is denoted by~$M_{i,j}$.

\looseness-1A sequence $\left\{x^l\right\}$ converges to $x^\ast$ at a~Q-linear rate $\varrho\in\left]0,1\right[$ if, for $l$ large enough,
\begin{align}
\nonumber
\displaystyle\frac{\left\|x^{l+1}-x^{\ast}\right\|_2}{\left\|x^l-x^{\ast}\right\|_2}\leq\varrho\enspace.
\end{align}
\looseness-1The convergence rate is said to be~Q-superlinear if the above ratio tends to zero as~$l$ goes to infinity.
\section{A Trust Region Algorithm with Distributed Activity Detection}
\label{sec:algo}
\subsection{Algorithm Formulation}
\label{subsec:desc_algo}
The problem we consider is that of minimising a partially separable objective function subject to separable convex constraints.
\begin{align}
\label{eq:bnd_nlp}
&\minimise_{w}~L\left(w_1,\ldots,w_N\right)\\
&\text{s.t.}~w_i\in\Wcal_i,~\forall i\in\left\{1,\ldots,N\right\}\enspace,\nonumber
\end{align}
where~$w:=\left(w_1^{\Trans},\ldots,w_N^{\Trans}\right)^{\Trans}\in\Rset^n$, with~$n=\sum_{i=1}^Nn_i$, and~$\Wcal:=\Wcal_1\times\ldots\times\Wcal_N$, where the sets~$\Wcal_i\subset\Rset^{n_i}$ are closed and convex.~The following Assumption is standard in distributed computations\ \citep{bert1997}.
\begin{ass}[Colouring scheme]
\label{ass:part_sep}
\looseness-1The sub-variables~$w_1,\ldots,w_N$ can be re-ordered and grouped together in such a way that a~Gauss-Seidel minimisation sweep on the function~$L$ can be performed in parallel within~$K\ll N$ groups, which are updated sequentially.~In the sequel, the re-ordered variable is denoted by~$x=\left(x_1^{\Trans},\ldots,x_K^{\Trans}\right)^{\Trans}$.~The set~$\Wcal$ is transformed accordingly into~$\Omega=\Omega_{1}\times\ldots\times\Omega_K$.~It is worth noting that each set~$\Omega_k$ with~$k\in\left\{1,\ldots,K\right\}$ can then be decomposed further into sets~$\Wcal_i$ with~$i\in\left\{1,\ldots,N\right\}$. 
\end{ass}
As a consequence of~Assumption\ \ref{ass:part_sep},~NLP\ \eqref{eq:bnd_nlp} is equivalent to
\begin{align}
&\minimise_x~L\left(x_1,\ldots,x_K\right)\nonumber\\
&\text{s.t.}~x_k\in\Omega_k,~\forall k\in\left\{1,\ldots,K\right\}\enspace.\nonumber 
\end{align}
\begin{rk}
\looseness-1Such a partially separable structure in the objective (Assumption\ \ref{ass:part_sep}) is encountered very often in practice, for instance when relaxing network coupling constraints via an augmented Lagrangian penalty. Thus, by relaxing the nonlinear coupling constraint $C\left(w_1,\ldots,w_N\right)=0$ and the local equality constraints $g_i\left(w_i\right)=0$ of
\begin{align}
\minimise_{w_1,\ldots,w_N}&\sum_{i=1}^Nf_i\left(w_i\right)\nonumber\\
\text{s.t.}~&C\left(w_1,\ldots,w_N\right)=0\nonumber\\
&~~~~~~~~~~~g_i\left(w_i\right)=0\nonumber\\
&~~~~~~~~~~~~~~~~w_i\in\Wcal_i\nonumber\\
&~~~~~~~i\in\left\{1,\ldots,N\right\}\enspace,\nonumber
\end{align}
in a differentiable penalty function, one obtains an NLP of the form\ \eqref{eq:bnd_nlp}.~In NLPs resulting from the direct transcription of optimal control problems, the objective is generally separable and the constraints are stage-wise with a coupling between the variables at a given time instant with the variables of the next time instant.~In this particular case,~the number of groups is $K=2$.~In Section\ \ref{sec:ac_pow}, we illustrate this property by means of examples arising from various formulations of the~Optimal Power Flow (OPF) problem.~The number of colours~$K$ represents the level of parallelism that can be achieved in a~Gauss-Seidel method for solving\ \eqref{eq:bnd_nlp}.~Thus, in the case of a discretised OCP, an alternating projected gradient sweep can be applied in two steps during which all updates are parallel. 
\end{rk}
\looseness-1For the sake of exposition, in order to make the distributed nature of our algorithm apparent, we assume that every sub-variable $w_i$, with $i\in\left\{1,\ldots,N\right\}$, is associated with a computing node.~Two nodes are called\ \textit{neighbours} if they are coupled in the objective $L$.~Our goal is to find a first-order critical point of NLP\ \eqref{eq:bnd_nlp} via an iterative procedure for which we are given an initial feasible point $x^0\in\Omega$.~The iterative method described next aims at computing every iterate in a distributed fashion, which requires communications between neighbouring nodes and leads to a significant level of concurrency.
\begin{ass}
\label{ass:bnd_blw}
The objective function~$L$ is bounded below on~$\left\{x\in\Omega~:~L(x)\leq L(x^0)\right\}$.
\end{ass}
The algorithm formulation can be done for any convex set~$\Omega$, but some features are more suitable for linear inequality constraints. 
\begin{ass}[Polyhedral constraints]
\label{ass:poly_set}
For all~$k\in\left\{1,\ldots,K\right\}$, the set~$\Omega_k$ is a non-empty polyhedron, such that
\begin{align}
\nonumber
\Omega_k:=\left\{x\in\Rset^{n_k}~:~\left<\omega_{k,i},x\right>\leq{}h_{k,i},~i\in\left\{1,\ldots,m_k\right\}\right\}\enspace,
\end{align}
with $\omega_{k,i}\in\Rset^{n_k}$, $h_{k,i}\in\Rset$ for all $i\in\left\{1,\ldots,m_k\right\}$ and $n_k, m_k\geq1$.
\end{ass}
\begin{ass}
\label{ass:smooth}
\looseness-1The objective function~$L$ is continuously differentiable in an open set containing~$\Omega$.~Its gradient $\nabla{}L$ is uniformly continuous.
\end{ass}
It is well-known\ \citep{conn2000} that for problem\ \eqref{eq:bnd_nlp},~$x^\ast$ being a critical point is equivalent to 
\begin{align}
\label{eq:crit_def_1}
P_\Omega\left(x^\ast-\nabla{}L\left(x^\ast\right)\right)=x^\ast\enspace.
\end{align}
Algorithm\ \ref{algo:dis_tr} below is designed to compute a critical point~$x^\ast$ of the function~$L+\iota_\Omega$.~It is essentially a two-phase approach, in which an active-set is first computed and then, a quadratic model is minimised approximately on the current active face.~Standard two-phase methods compute the active-set by means of a centralised projected search, updating all variables centrally.~More precisely, a model of the objective is minimised along the projected objective gradient, which yields the~Cauchy point.~The model decrease provided by the~Cauchy point is then enhanced in a refinement stage.
\begin{algorithm}[h!]
	\caption{\label{algo:dis_tr}Trust Region Algorithm with Alternating Projections~(\textsc{trap})}
	\begin{algorithmic}[1]
		\State \textbf{Parameters:} Initial trust region radius~$\Delta$, update parameters~$\sigma_1$,~$\sigma_2$ and~$\sigma_3$ such that~$0<\sigma_1<\sigma_2<1<\sigma_3$, test ratios~$\eta_1$ and~$\eta_2$ such that~$0<\eta_1<\eta_2<1$, coefficients~$\gamma_1\in\left]0,1\right[$ and~$\gamma_2>0$, termination tolerance~$\epsilon$.
		\State \textbf{Input:} Initial guess~$x$, projection operators~$\left\{P_{\Omega_k}\right\}_{k=1}^K$, objective function~$L$, objective gradient~$\nabla{}L$.
		\While{$\left\|P_\Omega\left(x-\nabla{}L\left(x\right)\right)-x\right\|_2>\epsilon$}
		\State \underline{\textbf{Distributed activity detection (alternating gradient projections)}}: \label{line:ac_detec_beg}
			\For{$k=1\ldots,K$}
				\State ${z_k\leftarrow{}P_{\Omega_k}\left(x_k-\alpha_k\nabla_km\left(z_{\left[\!\left[1,k-1\right]\!\right]},x_k,x_{\left[\!\left[k+1,K\right]\!\right]}\right)\right)}$,\ \Comment{In parallel in group~$k$}\\
				 ~~~~~~~~~where~$\alpha_k$ is computed according to requirements\ \eqref{eq:stp_small},\ \eqref{eq:stp_large_1} and\ \eqref{eq:stp_large_2}.
			\EndFor \label{line:ac_detec_end}
		\State \underline{\textbf{Distributed refinement (Algorithm\ \ref{algo:pcg})}}:\label{line:ref_beg} 
			\State Find~$y\in\Omega$ such that 
			\State ~~~~${m\left(x\right)-m\left(y\right)\geq\gamma_1\left(m\left(x\right)-m\left(z\right)\right)}$
			\State ~~~~$\left\|y-x\right\|_\infty\leq\gamma_2\Delta$ 
			\State ~~~~$\Acal_{\Omega_k}\left(z_k\right)\subset\Acal_{\Omega_k}\left(y_k\right)$~for all~$k\in\left\{1,\ldots,K\right\}$.\label{line:ref_end}
		\State \underline{\textbf{Trust-region update}}:
		\State $\rho\leftarrow\displaystyle\nicefrac{L\left(x\right)-L\left(y\right)}{m\left(x\right)-m\left(y\right)}$
		\If{$\rho<\eta_1$}\ \Comment{Not successful}\label{line:test_beg}
		\State \textit{(Do not update~$x$)}
		\State Take~$\Delta$ within~$\left[\sigma_1\Delta,\sigma_2\Delta\right]$
		\ElsIf{$\rho\in\left[\eta_1,\eta_2\right]$}\Comment{Successful}\label{line:succ_it}
		\State $x\leftarrow{}y$
		\State Take~$\Delta$ within~$\left[\sigma_1\Delta,\sigma_3\Delta\right]$
		\State Update objective gradient~$\nabla{}L\left(x\right)$ and model hessian~$B\left(x\right)$. 
		\Else \Comment{Very successful}\label{line:ver_succ_it}
		\State $x\leftarrow{}y$
		\State Take~$\Delta$ within~$\left[\Delta,\sigma_3\Delta\right]$
		\State Update objective gradient~$\nabla{}L\left(x\right)$ and model hessian~$B\left(x\right)$ 
		\EndIf\label{line:test_end}
		\EndWhile	
		\end{algorithmic}
\end{algorithm}
\looseness-1Similarly to a two-phase method, in order to globalise convergence,~Algorithm\ \ref{algo:dis_tr} uses the standard trust region mechanism.~At every iteration, a model~$m$ of the objective function~$L$ is constructed around the current iterate~$x$ as follows
\begin{align}
\label{eq:model}
m\left(x'\right):=L\left(x\right)+\left<\nabla{}L\left(x\right),x'-x\right>+\displaystyle\frac{1}{2}\left<x'-x,B\left(x\right)\left(x'-x\right)\right>\enspace,
\end{align}
where $x'\in\Rset^n$ and $B\left(x\right)$ is a symmetric matrix. 
\begin{ass}[Uniform bound on model hessian]
\label{ass:bnd_hss_mod}
There exists~$\hat{B}>0$ such that
\begin{align}
\left\|B\left(x\right)\right\|_2\leq\hat{B}\enspace,\nonumber
\end{align}
for all~$x\in\Omega$.
\end{ass}
The following Assumption is necessary to ensure distributed computations in Algorithm\ \ref{algo:dis_tr}.~It is specific to Algorithm\ \ref{algo:dis_tr} and does not appear in the standard trust region methods\ \citep{burke1990}.
\begin{ass}[Structured model hessian]
\label{ass:spar_hss}
For all~$x\in\Omega$, for all~$i,j\in\left\{1,\ldots,n\right\}$, 
\begin{align}
\nonumber
\nabla^2L_{i,j}\left(x\right)=0\Leftrightarrow{}B_{i,j}\left(x\right)=0\enspace.
\end{align}
\end{ass}
\begin{rk}
It is worth noting that the partial separability structure of the objective function $L$ is transferred to the sparsity pattern of its hessian $\nabla^2L$, hence, by~Assumption\ \ref{ass:spar_hss}, to the sparsity pattern of the model hessian $B$.~Hence, a~Gauss-Seidel sweep on the model function $m$ can also be carried out in $K$ parallel steps.
\end{rk}
\looseness-1~The main characteristic of\ \textsc{trap} is the activity detection phase, which differs from the projected search in standard trust region methods\ \citep{burke1990}.~At every iteration,\ \textsc{trap} updates the current active-set by computing iterates $z_1,\ldots,z_K$ (Lines\ \ref{line:ac_detec_beg} to\ \ref{line:ac_detec_end}).~This is the main novelty of\ \textsc{trap}, compared to existing two-phase techniques, and allows for different step-sizes $\alpha_1,\ldots,\alpha_K$ per block of variables, which is relevant in a distributed framework, as the current active-set can be split among nodes and does not need to be computed centrally.~In the trust region literature, the point
\begin{align}
\label{eq:def_cauch_pt}
z:=\left(z_1^{\Trans},\ldots,z_K^{\Trans}\right)^{\Trans}\enspace,
\end{align}
is often referred to as the\ \textit{Cauchy point}.~We keep this terminology in the remainder of the paper.~It is clear from its formulation that\ \textsc{trap} allows one to compute~Cauchy points via independent projected searches on every node.~Once the~Cauchy points $z_1,\ldots,z_K$ have been computed, they are used in the refinement step to compute a new iterate $y$ that satisfies the requirements shown from Lines\ \ref{line:ref_beg} to\ \ref{line:ref_end}.~The last step consists in checking if the model decrease $m(y)-m(x)$ is close enough to the variation in the objective $L$ (Lines\ \ref{line:test_beg} to\ \ref{line:test_end}).~In this case, the iterate is updated and the trust region radius $\Delta$ increased, otherwise the radius is shrunk and the iterate frozen.~This operation requires a global exchange of information between nodes.

In the remainder, the objective gradient $\nabla{}L\left(x\right)$ is denoted by $g\left(x\right)$ and the objective hessian $\nabla^2L\left(x\right)$ by $H\left(x\right)$.~The model function $m$ is an approximation of the objective function $L$ around the current iterate $x$.~The quality of the approximation is controlled by the trust region, defined as the box
\begin{align}
\nonumber
\mathbb{B}\left(x-\Delta,x+\Delta\right)\enspace, 
\end{align}
where~$\Delta$ is the trust region radius. 

\looseness-1In the rest of the paper, we denote the~Cauchy points by~$z_k$ or~$z_k\left(\alpha_k\right)$ without distinction, where~$\alpha_k$ are appropriately chosen step-sizes.~More precisely, following Section~$3$ in\ \citep{burke1990}, in\ \textsc{trap}, the block-coordinate step-sizes~$\alpha_k$ are chosen so that for all~$k\in\left\{1,\ldots,K\right\}$, the~Cauchy points~$z_k$ satisfy 
\begin{align}
\label{eq:stp_small}
\left\{
\begin{aligned}
&m\left(z_{\left[\!\left[1,k-1\right]\!\right]},z_k,x_{\left[\!\left[k+1,K\right]\!\right]}\right)\leq{}m\left(z_{\left[\!\left[1,k-1\right]\!\right]},x_k,x_{\left[\!\left[k+1,K\right]\!\right]}\right)\\
&~~~~~~~~~~~~~~~~~~~~~~~~~~~~~~~~~~~+\nu_0\left<\nabla_km\left(z_{\left[\!\left[1,k-1\right]\!\right]},x_k,x_{\left[\!\left[k+1,K\right]\!\right]}\right),z_k-x_k\right>\\
&\left\|z_k-x_k\right\|_\infty\leq\nu_2\Delta 
\end{aligned}\right.\enspace,
\end{align}
with~$\nu_0\in\left]0,1\right[$ and~$\nu_2>0$, where~$z_{\left[\!\left[1,k-1\right]\!\right]}$ stands for~$\left(z_1^{\Trans},\ldots,z_{k-1}^{\Trans}\right)^{\Trans}$, along with the condition that there exists positive scalars~$\nu_1<\nu_2$,~$\nu_3$,~$\nu_4$~and~$\nu_5$ for all~$k\in\left\{1,\ldots,K\right\}$, 
\begin{align}
\label{eq:stp_large_1}
\alpha_k\in\left[\nu_4,\nu_5\right]~~~~~\text{or}~~~~~\alpha_k\in\left[\nu_3\bar{\alpha}_k,\nu_5\right]
\end{align}
where the step-sizes $\bar{\alpha}_k$ are such that one of the following conditions hold for every $k\in\left\{1,\ldots,K\right\}$, 
\begin{align}
\label{eq:stp_large_2}
&m\left(z_{\left[\!\left[1,k-1\right]\!\right]},z_k\left(\bar{\alpha}_k\right),x_{\left[\!\left[k+1,K\right]\!\right]}\right)>m\left(z_{\left[\!\left[1,k-1\right]\!\right]},x_k,x_{\left[\!\left[k+1,K\right]\!\right]}\right)\\
&~~~~~~~~~~~~~~~~~~~~~~~~~~~~~~~~~~~~~~~~~+\nu_0\left<\nabla_km\left(z_{\left[\!\left[1,k-1\right]\!\right]},x_k,x_{\left[\!\left[k+1,K\right]\!\right]}\right),z_k\left(\bar{\alpha}_k\right)-x_k\right>\enspace,\nonumber
\end{align}
or
\begin{align}
\label{eq:stp_large_3}
\left\|z_k\left(\bar{\alpha}_k\right)-x_k\right\|_\infty\geq\nu_1\Delta\enspace,
\end{align}
\looseness-1~Conditions\ \eqref{eq:stp_small} ensure that the step-sizes~$\alpha_k$ are small enough to enforce a sufficient decrease coordinate-wise, as well as containment within a scaled trust region.~Conditions\ \eqref{eq:stp_large_1},\ \eqref{eq:stp_large_2} and\ \eqref{eq:stp_large_3} guarantee that the step-sizes~$\alpha_k$ do not become arbitrarily small.~All conditions\ \eqref{eq:stp_small},\ \eqref{eq:stp_large_1} and\ \eqref{eq:stp_large_2} can be tested in parallel in each of the~$K$ groups of variables.~In the next two paragraphs\ \ref{subsec:stp_size} and\ \ref{subsec:comp_cons} of this section, the choice of step-sizes~$\alpha_k$ ensuring the sufficient decrease is clarified, as well as the distributed refinement step.~In the next Section\ \ref{sec:cv_ana}, the convergence properties of\ \textsc{trap} are analysed.~Numerical examples are presented in Section\ \ref{sec:ac_pow}. 
 
\subsection{Step-sizes Computation in the Activity Detection Phase}
\label{subsec:stp_size}
\looseness-1At a given iteration of\ \textsc{trap}, the step-sizes~$\alpha_k$ are computed by backtracking to ensure a sufficient decrease at every block of variables and coordinate-wise containment in a scaled trust region as formalised by\ \eqref{eq:stp_small}.~It is worth noting that the coordinate-wise backtracking search can be run in parallel among the variables of group~$k$, as they are decoupled from each other.~As a result, there is one step-size per sub-variable~$w_i$ in group~$k$.~Yet, for simplicity, we write it as a single step-size~$\alpha_k$.~The reasoning of Section\ \ref{sec:cv_ana} can be adapted accordingly.~The following Lemma shows that a coordinate-wise step-size~$\alpha_k$ can be computed that ensures conditions\ \eqref{eq:stp_small},\ \eqref{eq:stp_large_1},\ \eqref{eq:stp_large_2} and\ \eqref{eq:stp_large_3} on every block of coordinates~$k\in\left\{1,\ldots,K\right\}$.
\begin{lem}
\label{lem:bck}
Assume that Assumption\ \ref{ass:bnd_hss_mod} holds.~For all~$k\in\left\{1,\ldots,K\right\}$, an iterate~$z_k$ satisfying conditions\ \eqref{eq:stp_small},\ \eqref{eq:stp_large_1},\ \eqref{eq:stp_large_2} and\ \eqref{eq:stp_large_3} can be found after a finite number of backtracking iterations.
\end{lem}
\begin{proof}
Let~$k\in\left\{1,\ldots,K\right\}$.~We first show that for a sufficiently small~$\alpha_k$, conditions\ \eqref{eq:stp_small} are satisfied.~By definition of the Cauchy point~$z_k$,  
\begin{align}
\nonumber
z_k=\argmin_{z\in\Omega_k}~\left<\nabla_km\left(z_{\left[\!\left[1,k-1\right]\!\right]},x_k,x_{\left[\!\left[k+1,K\right]\!\right]}\right),z-x_k\right>+\displaystyle\frac{1}{2\alpha_k}\left\|z-x_k\right\|_2^2\enspace,
\end{align}
which implies that 
\begin{align}
\nonumber
\left<\nabla_km\left(z_{\left[\!\left[1,k-1\right]\!\right]},x_k,x_{\left[\!\left[k+1,K\right]\!\right]}\right),z_k-x_k\right>+\displaystyle\frac{1}{2\alpha_k}\left\|z_k-x_k\right\|_2^2\leq0\enspace,
\end{align}
Hence, as~$\nu_0\in\left]0,1\right[$, it follows that 
\begin{align}
&\left<\nabla_km\left(z_{\left[\!\left[1,k-1\right]\!\right]},x_k,x_{\left[\!\left[k+1,K\right]\!\right]}\right),z_k-x_k\right>+\displaystyle\frac{1-\nu_0}{2\alpha_k}\left\|z_k-x_k\right\|_2^2\leq\nonumber\\
&~~~~~~~~~~~~~~~~~~~~~~~~~~~~~~~~~~~~~~~~~~~~~~~~~~~~~~~~\nu_0\left<\nabla_km\left(z_{\left[\!\left[1,k-1\right]\!\right]},x_k,x_{\left[\!\left[k+1,K\right]\!\right]}\right),z_k-x_k\right>\enspace.\nonumber
\end{align}
However, from the descent Lemma, which can be applied since the model gradient is Lipschitz continuous by Assumption\ \ref{ass:bnd_hss_mod}, 
\begin{align}
&m\left(z_{\left[\!\left[1,k-1\right]\!\right]},z_k,x_{\left[\!\left[k+1,K\right]\!\right]}\right)\leq{}m\left(z_{\left[\!\left[1,k-1\right]\!\right]},x_k,x_{\left[\!\left[k+1,K\right]\!\right]}\right)+\left<\nabla_km\left(z_{\left[\!\left[1,k-1\right]\!\right]},x_k,x_{\left[\!\left[k+1,K\right]\!\right]}\right),z_k-x_k\right>\nonumber\\
&~~~~~~~~~~~~~~~~~~~~~~~~~~~~~~~~~~~~~~~~~~~~~~~~~~~~~~~~~~~~~~~~~~~~~~+\displaystyle\frac{\hat{B}}{2}\left\|z_k-x_k\right\|_2^2\enspace.\nonumber
\end{align}
By choosing
\begin{align}
\nonumber
\alpha_k\leq\displaystyle\frac{1-\nu_0}{\hat{B}}\enspace,
\end{align}
condition\ \eqref{eq:stp_small} is satisfied after a finite number of backtracking iterations.~Denoting by~$q_k$ the smallest integer such that requirement\ \eqref{eq:stp_small} is met,~$\alpha_k$ can be written 
\begin{align}
\nonumber
\alpha_k=c^{q_k}\cdot\alpha^{(0)}\enspace,
\end{align}
where~$c\in\left]0,1\right[$ and~$\alpha^{(0)}>0$.~Then, condition\ \eqref{eq:stp_large_1} is satisfied with~$\nu_4=\alpha^{(0)}$ and~$\nu_3=c$.
\end{proof}
Lemma\ \ref{lem:bck} is very close to Theorem~$4.2$ in\ \citep{more1988}, but the argument regarding the existence of the step-sizes~$\alpha_k$ is different.

\subsection{Distributed Computations in the Refinement Step}
\label{subsec:comp_cons}
\looseness-1In Algorithm\ \ref{algo:dis_tr}, the objective gradient~$g\left(x\right)$ and model hessian~$B\left(x\right)$ are updated after every successful iteration.~This task requires exchanges of variables between neighbouring nodes, as the objective is partially separable (Ass.\ \ref{ass:part_sep}).~Node~$i$ only needs to store the sub-part of the objective function~$L$ that combines its variable~$w_i$ and the variables associated to its neighbours.~However, the refinement step (line\ \ref{line:ref_beg} to\ \ref{line:ref_end} in Algorithm\ \ref{algo:dis_tr}), in which one obtains a fraction of the model decrease yielded by the Cauchy points~$z_1,\ldots,z_K$, should also be computed in a distributed manner.~As detailed next, this phase consists in solving the~Newton problem on the subspace of free variables at the current iteration, which is defined as the set of free variables at the~Cauchy points~$z_1,\ldots,z_K$.~In order to achieve a reasonable level of efficiency in the trust region procedure, this step is generally performed via the~Steihaug-Toint~CG, or~sCG\ \citep{stei1983}.~The~sCG algorithm is a~CG procedure that is cut if a negative curvature direction is encountered or a problem bound is hit in the process.~Another way of improving on the~Cauchy point to obtain fast local convergence is the~Dogleg strategy\ \citep{noce2006}.~However, this technique requires the model hessian~$B$ to be positive definite\ \citep{noce2006}.~This condition does not fit well with distributed computations, as positive definiteness is typically enforced by means of BFGS updates, which are know for not preserving the sparsity structure of the objective without non-trivial modifications and assumptions\ \citep{yama2008}.~Compared to direct methods, iterative methods such as the~sCG procedure have clear advantages in a distributed framework, for they do not require assembling the hessian matrix on a central node.~Furthermore, their convergence speed can be enhanced via block-diagonal preconditioning, which is suitable for distributed computations.~In the sequel, we briefly show how a significant level of distributed operations can be obtained in the~sCG procedure, mainly due to the sparsity structure of the model hessian that matches the partial separability structure of the objective function.~More details on distributed implementations of the~CG algorithm can be found in numerous research papers\ \citep{verschoor2012,azevedo1993,fei2014}.~The~sCG algorithm that is described next is a rearrangement of the standard~sCG procedure following the idea of\ \citep{azevedo1993}.~The two separate inner products that usually appear in the~CG are grouped together at the same stage of the~Algorithm. 

\looseness-1An important feature of the refinement step is the increase of the active set at every iteration.~More precisely, in order to ensure finite detection of activity, the set of active constraints at the points~$y_1,\ldots,y_K$, obtained in the refinement phase, needs to contain the set of active constraints at the~Cauchy points~$z_1,\ldots,z_K$, as formalised at line\ \ref{line:ref_end} of~Algorithm\ \ref{algo:dis_tr}.~This requirement is very easy to fulfil when~$\Omega$ is a bound constraint set, as it just requires enforcing the constraint
\begin{align}
\nonumber
y_{k,i}=z_{k,i},~i\in\left\{j\in\left\{1,\ldots,n_k\right\}~:~z_{k,j}=\underline{x}_{k,j}\ \text{or}\ \bar{x}_{k,j}\right\}
\end{align}
for all groups~$k\in\left\{1,\ldots,K\right\}$ in the trust region problem at the refinement step.

\looseness-1For the convergence analysis that follows in Section\ \ref{sec:cv_ana}, the refinement step needs to be modified compared to existing trust region techniques.~Instead of solving the standard refinement problem
\begin{align}
\minimise_p~&\left<g\left(x\right),p\right>+\displaystyle\frac{1}{2}\left<p,B\left(x\right)p\right>\nonumber\\
\text{s.t.}~&\left\|p\right\|_\infty\leq\gamma_2\Delta\nonumber\\
&x+p\in\Omega\nonumber\\
&\Acal_\Omega\left(z\right)\subseteq\Acal_\Omega\left(x+p\right)\enspace,\nonumber
\end{align}
in which the variables corresponding to indices of active constraints at the~Cauchy point $z$ are fixed to zero, we solve a regularised version
\begin{align}
\label{eq:prox_ref}
\minimise_{y\in\Omega}&~\left<g\left(x\right),y-x\right>+\displaystyle\frac{1}{2}\left<y-x,B\left(x\right)\left(y-x\right)\right>+\displaystyle\frac{\sigma}{2}\left\|y-z\right\|_2^2\\
\text{s.t.}&~\left\|y-x\right\|_\infty\leq\gamma_2\Delta\nonumber\\
&\Acal_\Omega\left(z\right)\subseteq\Acal_\Omega\left(y\right)\enspace,\nonumber
\end{align}
\looseness-1where $\sigma\in\left]\underline{\sigma},\bar{\sigma}\right[$ with $\underline{\sigma}>0$, and $z$ is the~Cauchy point yielded by the procedure described in the previous paragraph\ \ref{subsec:stp_size}.~The regularisation coefficient $\sigma$ should not be chosen arbitrarily, as it may inhibit the fast local convergence properties of the~Newton method.~This point is made explicit in paragraph\ \ref{subsec:cv_rate} of Section\ \ref{sec:cv_ana}.~The regularised trust region subproblem\ \eqref{eq:prox_ref} can be equivalently written
\begin{align}
\label{eq:prox_ref_eq}
\minimise_{p}&~\left<g_{\sigma}\left(x\right),p\right>+\displaystyle\frac{1}{2}\left<p,B_{\sigma}\left(x\right)p\right>\\
\text{s.t.}&~x+p\in\Omega\nonumber\\
&~\left\|p\right\|_\infty\leq\gamma_2\Delta\nonumber\\
&\Acal_\Omega\left(z\right)\subseteq\Acal_\Omega\left(x+p\right)\enspace,\nonumber
\end{align}
with
\begin{align}
\label{eq:reg_g_hss}
g_{\sigma}\left(x\right):=g\left(x\right)-\sigma(z-x),~~B_{\sigma}\left(x\right):=B\left(x\right)+\displaystyle\frac{\sigma}{2}I\enspace.
\end{align}
As in standard trust region methods, we solve the refinement subproblem\ \eqref{eq:prox_ref_eq} by means of CG iterations, which can be distributed as a result of Assumption\ \ref{ass:spar_hss}.~In order to describe this stage in Algorithm\ \ref{algo:pcg}, one needs to assume that $\Omega$ is a box constraint set.~In the remainder, we denote by $Z$ the matrix whose columns are an orthonormal basis of the subspace
\begin{align}
\nonumber
V\left(z\right):=\left\{x\in\Rset^n~:~\left<\omega_{k,i},x_k\right>=0,~i\in\Acal_{\Omega_k}\left(z_k\right),~k\in\left\{1,\ldots,K\right\}\right\}\enspace.
\end{align}
\begin{algorithm}[h!]
	\caption{\label{algo:pcg}Distributed Safeguarded Conjugate Gradient (sCG)}
	\begin{algorithmic}[1]
	\State \textbf{Input:}~reduced model hessian~$\hat{B}_{\sigma}:=Z^{\Trans}B_{\sigma}Z$,~reduced gradient~$\hat{g}:=Z^{\Trans}g$,~initial guess~$\hat{z}:=Z^{\Trans}z$
	\State \textbf{Parameters:}~stopping tolerance~$\hat{\epsilon}:=\xi\left\|\hat{g}\right\|_2$ with~$\xi\in\left]0,1\right[$
	\State Initialise~$\hat{x}$,~$\hat{p}$,~$\hat{v}$,~$\hat{r}$,~$\hat{t}$ and~$\hat{u}_{\text{prev}}$ via a standard sCG iteration using~$z$,~$x$,~$Z$,~$B_{\sigma}$,~$\hat{B}_{\sigma}$ and~$\hat{g}$
	\While{$\hat{u}>\hat{\epsilon}^2$ and~$\hat{t}>0$}
	\State Compute structured matrix-vector product~$\hat{s}\leftarrow{}\hat{B}_{\sigma}\hat{r}$\ \Comment{Local communications}\label{pcg:prod}
	\For{$k=1\ldots K$}\ \Comment{In parallel among $K$ groups}\label{pcg:beg_inprd}
	\State Compute~$\left<\hat{r}_k,\hat{r}_k\right>$ and~$\left<\hat{r}_k,\hat{s}_k\right>$ 
	\EndFor\label{pcg:end_inprd}
	\State $\hat{u}\leftarrow\sum_{i=k}^K\left<\hat{r}_k,\hat{r}_k\right>$,~$\hat{\delta}\leftarrow\sum_{k=1}^K\left<\hat{r}_k,\hat{s}_k\right>$\ \Comment{Global summations}\label{pcg:glob_sum}
	\State Compute step-sizes~$\hat{\beta}\leftarrow\displaystyle\nicefrac{\hat{u}}{\hat{u}_{\text{prev}}}$ and~$\hat{t}\leftarrow\hat{\delta}-\hat{\beta}^2\hat{t}$
	\For{$k=1\ldots K$}\ \Comment{In parallel among~$K$ groups}\label{pcg:beg_up_dir}
	\State Update conjugate direction~$\hat{p}_k\leftarrow\hat{r}_k+\hat{\beta}\hat{p}_k$ and~$\hat{v}_k\leftarrow\hat{s}_k+\hat{\beta}\hat{v}_k$
	\State Compute smallest step-size~$a_k$ such that~$\hat{x}_k+a_k\hat{p}_k$ hits a bound~$\underline{x}_k$,~$\bar{x}_k$ or the trust region boundary
	\EndFor\label{pcg:end_up_dir}
	\If{$\hat{t}\leq0$}\ \Comment{Negative curvature check}
	\State Compute step-size~$\hat{a}\leftarrow\min\left\{a_1,\ldots,a_K\right\}$ to hit boundary of~$\mathbb{B}\left(x-\Delta,x+\Delta\right)\cap\Omega$ \label{pcg:stp_sz}
	\Else
	\State Compute standard CG step-size~$\hat{a}\leftarrow\displaystyle\nicefrac{\hat{u}}{\hat{t}}$
	\EndIf
	\For{$k=1\ldots K$}\ \Comment{In parallel among~$K$ groups}\label{pcg:beg_up_it}
	\State Update iterate~$\hat{x}_k\leftarrow\hat{x}_k+\hat{a}\hat{p}_k$ and residual~$\hat{r}_k\leftarrow\hat{r}_k-\hat{a}\hat{v}_k$
	\EndFor\label{pcg:end_up_it}
	\State $\hat{u}_{\text{prev}}\leftarrow\hat{u}$
	\EndWhile
	\State \textbf{Output:}~$y\leftarrow{}z+Z(\hat{x}-\hat{z})$
	\end{algorithmic}
\end{algorithm}
\begin{rk}
\looseness-1It is worth noting that the requirement~$m(x)-m(y)\geq\gamma_1\left(m(x)-m(z)\right)$, with~$\gamma_1<1$, is satisfied after any iteration of Algorithm\ \ref{algo:pcg}, as the initial guess is the~Cauchy point~$z$ and the~sCG iterations ensure monotonic decrease of the regularised model (Theorem~$2.1$ in\ \citep{stei1983}).
\end{rk} 
\begin{rk}
\looseness-1It is worth noting that the sparsity pattern of the reduced model hessian~$\hat{B}_{\sigma}$ has the same structure as the sparsity pattern of the model hessian~$B$, as the selection matrix~$Z$ has a block-diagonal structure.~Moreover, the partial separability structure of the objective matches the sparsity patterns of both the hessian and the reduced hessian.~For notational convenience, Algorithm\ \ref{algo:pcg} is written in terms of variables~$x_1,\ldots,x_K$, but it is effectively implementable in terms of variables~$w_1,\ldots,w_N$.~The inner products (Lines\ \ref{pcg:beg_inprd} to\ \ref{pcg:end_inprd}) and updates (Lines\ \ref{pcg:beg_up_dir} to\ \ref{pcg:end_up_dir}, lines\ \ref{pcg:beg_up_it} to\ \ref{pcg:end_up_it}) can be computed in parallel at every node, as well as the structured matrix-vector product (Line\ \ref{pcg:prod}).    
\end{rk} 
\looseness-1In Algorithm\ \ref{algo:pcg}, the reduced model hessian~$\hat{B}$ can be evaluated when computing the product at line\ \ref{pcg:prod}, which requires local exchanges of vectors between neighbouring nodes, since the sparsity pattern of~$\hat{B}$ represents the coupling structure in the objective~$L$.~From a distributed implementation perspective, the more costly parts of the refinement procedure\ \ref{algo:pcg} are at line\ \ref{pcg:glob_sum} and line\ \ref{pcg:stp_sz}.~These operations consist in summing up the inner products from all nodes and a minimum search over the step-sizes that ensure constraint satisfaction and containment in the trust region.~They need to be performed on a central node that has access to all data from other nodes, or via a consensus algorithm.~Therefore, lines\ \ref{pcg:glob_sum} and\ \ref{pcg:stp_sz} come with a communication cost, although the amount of transmitted data is very small (one scalar per node).
\begin{figure}[h!]
\begin{center}
\includegraphics[scale=0.35]{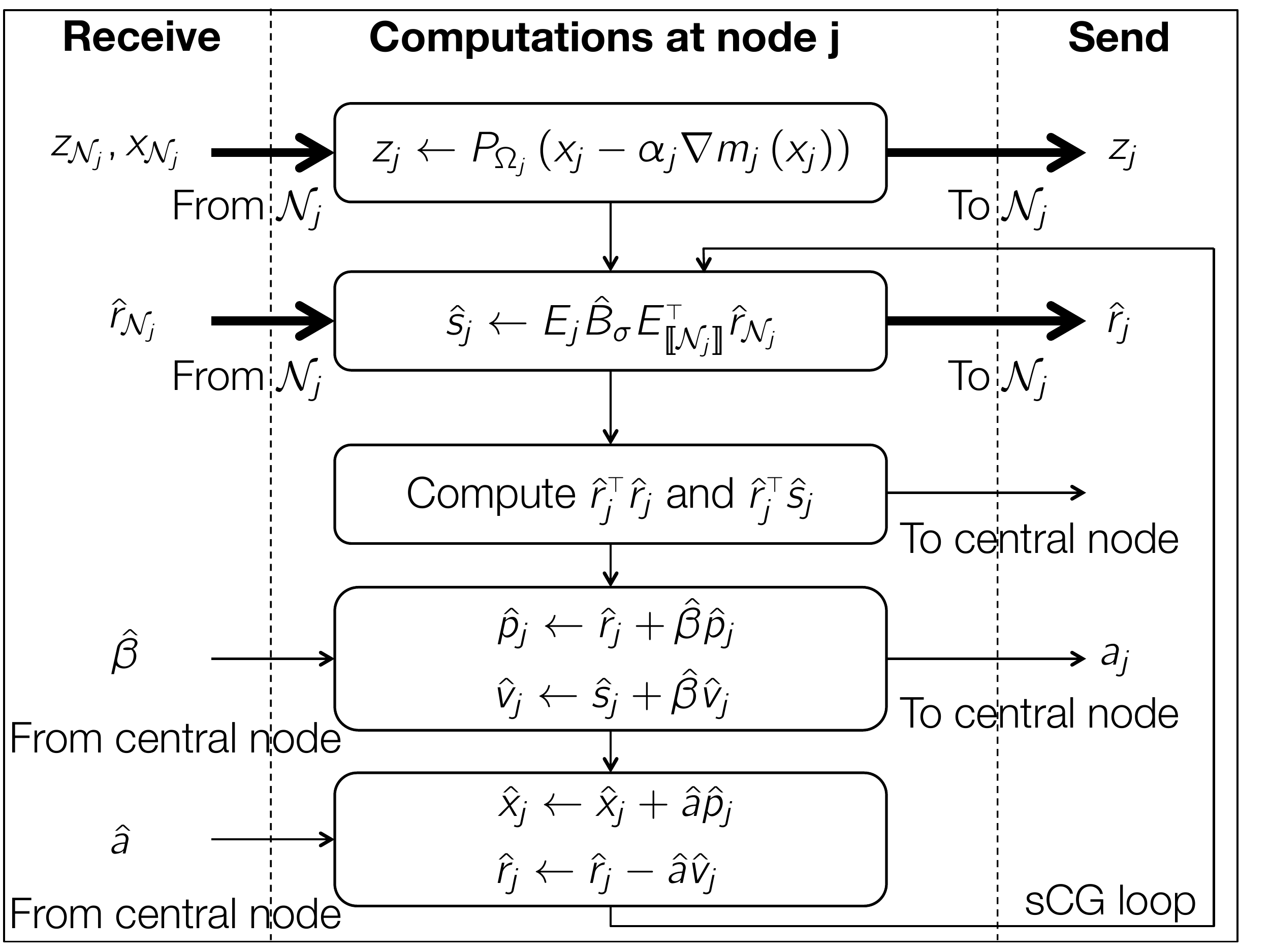}
\end{center}
\caption{\label{fig:wrk_flw}Workflow at node $j$ in terms of local computations, communications with the set of neighbours $\Ncal_j$ and a central node.~Note that we use the index $j$ for a node, and not $k$, which corresponds to a group of nodes, in which computations are performed in parallel.~Thus, the nodes in the set $\Ncal_j$ are not in the same group as node $j$.~Thick arrows represent communications involving vectors, whereas thin arrows stand for communications of scalars.~Matrix $E_j$ is defined at Eq.\ \eqref{eq:def_E_k}.}
\end{figure}
\looseness-1In the end, one should notice that the information that is required to be known globally by all nodes~$\left\{1,\ldots,N\right\}$ is fairly limited at every iteration of\ \textsc{trap}.~It only consists of the trust region radius~$\Delta$ and the step-sizes~$\hat{a}$ and~$\hat{\beta}$ in the refinement step\ \ref{algo:pcg}.~Finally, at every iteration, all nodes need to be informed of the success or failure of the iteration so as to update or freeze their local variables.~This is the result of the trust region test, which needs to be carried out on a central node.~In Figure\ \ref{fig:wrk_flw}, we give a sketch of the workflow at a generic node~$j$.~One can notice that, in terms of local computations,\ \textsc{trap} behaves as a standard two-phase approach on every node.
\section{Convergence Analysis}
\label{sec:cv_ana}
\looseness-1The analysis of\ \textsc{trap} that follows is along the lines of the convergence proof of trust region methods in\ \citep{burke1990}, where the Cauchy point is computed via a projected search, which involves a sequence of evaluations of the model function on a central node.~However, for\ \textsc{trap}, the fact that the Cauchy point is yielded by an distributed projected gradient step on the model function requires some modifications in the analysis.~Namely, the lower bound on the decrease in the model and the upper bound on criticality at the Cauchy point are expressed in a rather different way.~However, the arguments behind the global convergence proof are essentially the same as in\ \citep{burke1990}.

\looseness-1In this section, for theoretical purposes only, another first-order criticality measure different from\ \eqref{eq:crit_def_1} is used.~We utilise the condition that~$x^\ast\in\Omega$ is a first-order critical point if the projected gradient at~$x^\ast$ is zero, 
\begin{align}
\label{eq:crit_def_2}
\nabla_{\Omega}L\left(x^\ast\right)=0\enspace,
\end{align}
where, given~$x\in\Omega$, the projected gradient is defined as
\begin{align}
\nonumber
\nabla_{\Omega}L\left(x\right):=P_{\Tcal_\Omega\left(x\right)}\left(-g\left(x\right)\right)\enspace.
\end{align}
Discussions on this first-order criticality measure can be found in\ \citep{conn2000}.~It is equivalent to the standard optimality condition
\begin{align}
\nonumber
\left<g\left(x^\ast\right),x-x^\ast\right>\geq0,~\text{for all}~x\in\Omega\enspace.
\end{align}
It follows from~Moreau's decomposition that a point~$x^{\ast}$ satisfying\ \eqref{eq:crit_def_2} automatically satisfies\ \eqref{eq:crit_def_1}.~Consequently, it is perfectly valid to use\ \eqref{eq:crit_def_1} for the convergence analysis of\ \textsc{trap}.

\subsection{Global Convergence to First-order Critical Points}
\label{subsec:glob_cv}
\looseness-1We start with an estimate of the block-coordinate model decrease provided by the Cauchy points~$z_k$,~$k\in\left\{1,\ldots,K\right\}$, of Algorithm\ \ref{algo:dis_tr}.~For this purpose, we define for all~$k\in\left\{1,\ldots,K\right\}$,
\begin{align}
\label{eq:def_mod_coor}
m_k\left(x'\right):=m\left(z_{\left[\!\left[1,k-1\right]\!\right]},x',x_{\left[\!\left[k+1,K\right]\!\right]}\right)\enspace,
\end{align}
\looseness-1where~$x'\in\Rset^{n_k}$.~This corresponds to the model function evaluated at~$x'$ with the block-coordinates~$1$ to~$k-1$ being fixed to the associated Cauchy points~$z_1,\ldots,z_{k-1}$ and the block-coordinates~$k+1$ to~$K$ having values~$x_{k+1},\ldots,x_K$.~Note that by definition of the function~$m_k$,
\begin{align}
m_k\left(z_k\right)=m_{k+1}\left(x_{k+1}\right)\enspace,\nonumber
\end{align}
for all~$k\in\left\{1,\ldots,K-1\right\}$.
\begin{lem} 
\label{lem:blk_suff_dec}
There exists a constant~$\chi>0$ such that, for all~$k\in\left\{1,\ldots,K\right\}$,
\begin{align}
\label{eq:coor_dec}
m_k\left(x_k\right)-m_k\left(z_k\right)\geq\displaystyle\chi\frac{\left\|z_k-x_k\right\|_2}{\alpha_k}\min\left\{\Delta,\displaystyle\frac{1}{1+\left\|B(x)\right\|_2}\displaystyle\frac{\left\|z_k-x_k\right\|_2}{\alpha_k}\right\}\enspace.
\end{align}
\end{lem}
\begin{proof}
The proof goes along the same lines as the one of~Theorem~$4.3$ in\ \citep{more1988}.~Yet, some arguments differ, due to the alternating projections.~We first assume that condition\ \eqref{eq:stp_small} is satisfied with
\begin{align}
\nonumber
\alpha_k\geq\nu_4\enspace.
\end{align}
Using the basic property of the projection onto a closed convex set, we obtain
\begin{align}
m_k\left(x_k\right)-m_k\left(z_k\right)\geq\nu_0\nu_4\displaystyle\frac{\left\|z_k-x_k\right\|_2^2}{\alpha_k^2}\enspace.\nonumber
\end{align}
We then consider the second case when 
\begin{align}
\nonumber
\alpha_k\geq\nu_3\bar{\alpha}_k\enspace.
\end{align}
The first possibility is then 
\begin{align}
m_k\left(z_k\left(\bar{\alpha}_k\right)\right)-m_k\left(x_k\right)>\nu_0\left<\nabla{}m_k\left(x_k\right),z_k\left(\bar{\alpha}_k\right)-x_k\right>\enspace.\nonumber
\end{align}
However, by definition of the model function in~Eq.\ \eqref{line:ref_end}, the left-hand side term in the above inequality is equal to
\begin{align}
&\left<g_k\left(x\right),\bar{z}_k-x_k\right>+\displaystyle\frac{1}{2}\left<\bar{z}_k-x_k,E_kB\left(x\right)E_k^{\Trans}\left(\bar{z}_k-x_k\right)\right>\nonumber\\
&~~~~~~~~~~~~~~~~~~~~~~~~~~~~~~~~~~~~~~~~~~~~~~~~~~+\left<\bar{z}_{\left[\!\left[1,k-1\right]\!\right]}-x_{\left[\!\left[1,k-1\right]\!\right]},E_{\left[\!\left[1,k-1\right]\!\right]}B\left(x\right)E_k^{\Trans}\left(\bar{z}_k-x_k\right)\right>\nonumber\\
&=\displaystyle\frac{1}{2}\left<\bar{z}_k-x_k,E_kB\left(x\right)E_k^{\Trans}\left(\bar{z}_k-x_k\right)\right>+\left<\nabla{}m_k\left(x_k\right),\bar{z}_k-x_k\right>\enspace,\nonumber
\end{align} 
where, given~$k\in\left\{1,\ldots,K\right\}$, the matrix~$E_k\in\Rset^{n_k\times{}n}$ is such that for~$i\in\left\{1,\ldots,n_k\right\}$,
\begin{align}
\label{eq:def_E_k}
E_k(i,n_1+\ldots+n_{k-1}+i)=1\enspace,
\end{align}
and all other entries are zero.~This yields, by the~Cauchy-Schwarz inequality
\begin{align}
\displaystyle\frac{\left\|B\left(x\right)\right\|_2}{2}\left\|\bar{z}_k-x_k\right\|_2^2&>-\left(1-\nu_0\right)\left<\nabla{}m_k\left(x_k\right),\bar{z}_k-x_k\right>\nonumber\\
&\geq\displaystyle\frac{1-\nu_0}{\bar{\alpha}_k}\left\|\bar{z}_k-x_k\right\|_2^2\nonumber
\end{align}
Hence,
\begin{align}
\nonumber
\bar{\alpha}_k\geq\displaystyle\frac{2\left(1-\nu_0\right)}{1+\left\|B\left(x\right)\right\|_2}\enspace.
\end{align}
The second possibility is 
\begin{align}
\nonumber
\left\|\bar{z}_k-x_k\right\|_{\infty}\geq\nu_1\Delta\enspace.
\end{align}
For this case,\ \citep{more1988} provides the lower bound
\begin{align}
\nonumber
\left\|z_k-x_k\right\|_2\geq\nu_3\nu_1\Delta\enspace.
\end{align}
Finally, inequality\ \eqref{eq:coor_dec} holds with 
\begin{align}
\nonumber
\chi:=\nu_0\min\left\{\nu_4,2\left(1-\nu_0\right),\nu_3\nu_1\right\}\enspace.
\end{align}
\end{proof}
From Lemma\ \ref{lem:blk_suff_dec}, an estimate of the decrease in the model provided by the Cauchy point~$z$ is derived.
\begin{cor}[Sufficient decrease]
\label{cor:mod_suff_dec}
The following inequality holds
\begin{align}
\label{eq:mod_suff_dec}
&m\left(x\right)-m\left(z\right)\geq\chi\sum_{k=1}^K\displaystyle\frac{\left\|z_k-x_k\right\|_2}{\alpha_k}\min\left\{\Delta,\displaystyle\frac{1}{1+\left\|B\left(x\right)\right\|_2}\displaystyle\frac{\left\|z_k-x_k\right\|_2}{\alpha_k}\right\}\enspace.
\end{align}
\end{cor}
\begin{proof}
This is a direct consequence of Lemma\ \ref{lem:blk_suff_dec} above, as
\begin{align}
m\left(x\right)-m\left(z\right)&=\sum_{k=1}^{K}m_k\left(x_k\right)-m_k\left(z_k\right)\enspace.\nonumber
\end{align}
from the definition of~$m_k$ in Eq.\ \eqref{eq:def_mod_coor}.
\end{proof}
\looseness-1In a similar manner to\ \citep{burke1990}, the level of criticality reached by the Cauchy point~$z$ is measured by the norm of the projected gradient of the objective, which can be upper bounded by the difference between the current iterate~$x$ and the Cauchy point~$z$.   
\begin{lem}[Relative error condition]
\label{lem:coor_opt_upbnd}
The following inequality holds
\begin{align}
\left\|\nabla_{\Omega}L\left(z\right)\right\|_2\leq&~K\left\|B\left(x\right)\right\|_2\left\|z-x\right\|_2+\sum_{k=1}^K\left(\displaystyle\frac{\left\|z_k-x_k\right\|_2}{\alpha_k}+\left\|g_k\left(z\right)-g_k\left(x\right)\right\|_2\right)\enspace.
\label{eq:up_bnd_crit}
\end{align}
\end{lem}
\begin{proof}
From the definition of~$z_k$ as the projection of 
\begin{align}
x_k-\alpha_k\nabla{}m_k\left(x_k\right)\nonumber
\end{align}
onto the closed convex set~$\Omega_k$, there exists~$v_k\in\Ncal_{\Omega_k}\left(z_k\right)$ such that
\begin{align}
0&=v_k+\nabla{}m_k\left(x_k\right)+\displaystyle\frac{z_k-x_k}{\alpha_k}\enspace.\nonumber
\end{align}
Hence, 
\begin{align}
\left\|v_k+g_k\left(z\right)\right\|_2\leq&~\left\|g_k\left(z\right)-g_k\left(x\right)\right\|_2+\left\|B\left(x\right)\right\|_2\left\|z-x\right\|_2+\displaystyle\frac{\left\|z_k-x_k\right\|_2}{\alpha_k}\nonumber
\end{align}
However,
\begin{align}
\nonumber
\left\|P_{\Ncal_{\Omega_k}\left(z_k\right)}\left(-g_k\left(z\right)\right)+g_k\left(z\right)\right\|_2\leq\left\|v_k+g_k\left(z\right)\right\|_2\enspace,\nonumber
\end{align}
and by Moreau's decomposition theorem,
\begin{align}
-g_k\left(z\right)=P_{\Ncal_{\Omega_k}\left(z_k\right)}\left(-g_k\left(z\right)\right)+P_{\Tcal_{\Omega_k}\left(z_k\right)}\left(-g_k\left(z\right)\right)\enspace.\nonumber
\end{align}
Thus,
\begin{align}
\left\|P_{\Tcal_{\Omega_k}\left(z_k\right)}\left(-g_k\left(z\right)\right)\right\|_2\leq\left\|g_k\left(z\right)-g_k\left(x\right)\right\|_2+\left\|B\left(x\right)\right\|_2\left\|z-x\right\|_2+\displaystyle\frac{\left\|z_k-x_k\right\|_2}{\alpha_k}\enspace.\nonumber
\end{align}
As the sets~$\left\{\Omega_k\right\}_{k=1}^K$ are closed and convex, 
\begin{align}
\Tcal_{\Omega}\left(z\right)=\Tcal_{\Omega_1}\left(z_1\right)\times\ldots\times\Tcal_{\Omega_K}\left(z_K\right)\enspace.\nonumber
\end{align}
Subsequently,
\begin{align}
\left\|\nabla_\Omega{}L\left(z\right)\right\|_2\leq\sum_{k=1}^K\left\|P_{\Tcal_{\Omega_k}\left(z_k\right)}\left(-g_k\left(z\right)\right)\right\|_2\nonumber
\end{align}
and inequality\ \eqref{eq:up_bnd_crit} follows.
\end{proof}
\looseness-1Based on the estimate of the model decrease\ \eqref{eq:mod_suff_dec} and the relative error bound\ \eqref{eq:up_bnd_crit} at the Cauchy point~$z$, one can follow the standard proof mechanism of trust region methods quite closely\ \citep{burke1990}.~Most of the steps are proven by contradiction, assuming that criticality is not reached.~The nature of the model decrease\ \eqref{eq:mod_suff_dec} is well-suited to this type of reasoning.~Hence, most of the ideas of\ \citep{burke1990} can be adapted to our setting.
\begin{lem}
\label{lem:lim_inf}
If Assumptions\ \ref{ass:bnd_blw},\ \ref{ass:smooth} and\ \ref{ass:bnd_hss_mod} are satisfied, then the sequence of iterates yielded by Algorithm\ \ref{algo:dis_tr} satisfies that for all~$k\in\left\{1,\ldots,K\right\}$,
\begin{align}
\label{eq:lim_inf}
\liminf\displaystyle\frac{\left\|z_k-x_k\right\|_2}{\alpha_k}=0\enspace,
\end{align}
\end{lem}
\begin{proof}
For the sake of contradiction, assume that there exists a block index~$k_0\in\left\{1,\ldots,K\right\}$ and~$\epsilon>0$ such that
\begin{align}
\displaystyle\frac{\left\|z_{k_0}^l-x_{k_0}^l\right\|_2}{\alpha_{k_0}^l}\geq\epsilon\nonumber
\end{align}
for all iteration indices~$l\geq1$.~Using Corollary\ \ref{cor:mod_suff_dec}, the standard proof mechanism of trust region methods\ \citep{burke1990} can be easily adapted to obtain\ \eqref{eq:lim_inf}. 
\end{proof}
\looseness-1We are now ready to state the main Theorem of this section.~It is claimed that all limit points of the sequence~$\left\{x^l\right\}$ generated by\ \textsc{trap} are critical points of\ \eqref{eq:bnd_nlp}.
\begin{theo}[Limit points are critical points]
\label{th:cauch_cv}
Assume that Assumptions\ \ref{ass:bnd_blw},\ \ref{ass:smooth} and\ \ref{ass:bnd_hss_mod} hold.~If~$x^{\ast}$ is a limit point of~$\left\{x^l\right\}$, then there exists a subsequence~$\left\{l_i\right\}$ such that
\begin{align}
\label{eq:cv_crit}
\left\{ 
\begin{aligned}
&\lim_{i\rightarrow+\infty}\left\|\nabla_{\Omega}L\left(z^{l_i}\right)\right\|_2=0\\
&z^{l_i}\rightarrow{}x^{\ast}
\end{aligned}\right.\enspace.
\end{align}
Moreover,~$\nabla_{\Omega}L\left(x^\ast\right)=0$, meaning that~$x^{\ast}$ is a critical point of~$L+\iota_\Omega$.
\end{theo}
\begin{proof}
Let~$\left\{x^{l_i}\right\}$ be a subsequence of~$\left\{x^l\right\}$ such that~$x^{l_i}\rightarrow{}x^{\ast}$.~If for all~$k\in\left\{1,\ldots,K\right\}$ 
\begin{align}
\label{eq:sub_liminf}
\displaystyle\frac{\left\|z_k^{l_i}-x_k^{l_i}\right\|_2}{\alpha_{k}^{l_i}}\rightarrow0\enspace,
\end{align}
then the proof is complete, via Lemma\ \ref{lem:coor_opt_upbnd} and the fact that the step-sizes~$\alpha_k$ are upper bounded by~$\nu_5$.~In order to show\ \eqref{eq:sub_liminf}, given~$\epsilon>0$ one can assume that there exists~$k_0\in\left\{1,\ldots,K\right\}$ such that for all~$i\geq1$, $\displaystyle\nicefrac{\left\|z_{k_0}^{l_i}-x_{k_0}^{l_i}\right\|_2}{\alpha_{k_0}^{l_i}}\geq\epsilon$.~One can then easily combine the arguments in the proof of Theorem~$5.4$ in\ \citep{burke1990} with Corollary\ \ref{cor:mod_suff_dec} and Lemma\ \ref{lem:coor_opt_upbnd} in order to obtain\ \eqref{eq:cv_crit}.
\end{proof}
Theorem\ \ref{th:cauch_cv} above proves that all limit points of the sequence~$\left\{x^l\right\}$ generated by\ \textsc{trap} are critical points.~It does not actually claim convergence of~$\left\{x^l\right\}$ to a single critical point.~However, such a result can be obtained under standard regularity assumptions\ \citep{noce2006}, which ensure that a critical point is an isolated local minimum.
\begin{ass}[Strong second-order optimality condition]
\label{ass:reg_lim_pt}
The sequence~$\left\{x^l\right\}$ yielded by\ \textsc{trap} has a non-degenerate limit point~$x^{\ast}$ such that for all~$v\in\Ncal_{\Omega}\left(x^{\ast}\right)^{\perp}$, where
\begin{align}
\label{eq:def_nc_perp}
\Ncal_{\Omega}\left(x^{\ast}\right)^{\perp}:=\left\{v\in\Rset^n~:~\forall{}w\in\Ncal_{\Omega}\left(x^{\ast}\right),~\right<w,v\left>=0\right\}\enspace,
\end{align}
one has
\begin{align}
\label{eq:ssoc}
\left<v,H\left(x^{\ast}\right)v\right>\geq\kappa\left\|v\right\|_2^2\enspace,
\end{align}
where~$\kappa>0$.
\end{ass}
\begin{theo}[Convergence to first-order critical points]
\label{th:full_cv}
If Assumptions\ \ref{ass:bnd_hss_mod},\ \ref{ass:bnd_blw},\ \ref{ass:smooth} and\ \ref{ass:reg_lim_pt} are fulfilled, then the sequence~$\left\{x^l\right\}$ generated by\ \textsc{trap} converges to a non-degenerate critical point~$x^{\ast}$ of~$L+\iota_\Omega$. 
\end{theo}
\begin{proof}
\looseness-1This is an immediate consequence of Corollary~$6.7$ in\ \citep{burke1990}.
\end{proof}
\subsection{Active-set Identification}
\label{subsec:activ}
\looseness-1In most of the trust region algorithms for constrained optimisation, the Cauchy point acts as a predictor of the set of active constraints at a critical point.~Therefore, a desirable feature of the novel Cauchy point computation in\ \textsc{trap} is finite detection of activity, meaning that the active set at the limit point is identified after a finite number of iterations.~In this paragraph, we show that\ \textsc{trap} is equivalent to the standard projected search in terms of identifying the active set at the critical point~$x^{\ast}$ defined in Theorem\ \ref{th:full_cv}.
\begin{lem}
\label{lem:face}
Given a face~$\Fcal$ of~$\Omega$, there exists faces~$\Fcal_1,\ldots,\Fcal_K$ of~$\Omega_1,\ldots,\Omega_K$ respectively, such that~$\Fcal=\Fcal_1\times\ldots\times\Fcal_K$.
\end{lem}
\begin{rk}
Given a point~$x\in\Omega$, there exists a face~$\Fcal$ of~$\Omega$ such that~$x\in\ri\left(\Fcal\right)$.~The normal cone to~$\Omega$ at~$x$ is the cone generated by the normal vectors to the active constraints at~$x$.~As the set of active constraints is constant on the relative interior of a face, one can write without distinction~$\Ncal_\Omega\left(x\right)~\text{or}~\Ncal\left(\Fcal\right)$.
\end{rk}
\looseness-1The following Lemma is similar in nature to Lemma $7.1$ in\ \citep{burke1990}, yet with an adaptation in order to account for the novel way of computing the Cauchy point.~In particular, it is only valid for a sufficiently high iteration count, contrary to Lemma $7.1$ of\ \citep{burke1990}, which can be written independently of the iteration count.~This is essentially due to the fact that the~Cauchy point is computed via an alternating projected search, contrary to\ \citep{burke1990}, where a centralised projected search is performed.
\begin{lem}
\label{lem:ac_detec}
Assume that Assumptions\ \ref{ass:bnd_hss_mod},\ \ref{ass:bnd_blw},\ \ref{ass:smooth} and\ \ref{ass:reg_lim_pt} hold.~Let~$x^\ast$ be a non-degenerate critical point of\ \eqref{eq:bnd_nlp} that belongs to the relative interior of a face~$\Fcal^{\ast}$ of~$\Omega$.~Let~$\left\{\Fcal_k^{\ast}\right\}_{k=1}^K$ be faces of~$\left\{\Omega_k\right\}_{k=1}^K$ such that~$\Fcal^{\ast}=\Fcal_1^{\ast}\times\ldots\times\Fcal_K^{\ast}$ and thus~$x_k^{\ast}\in\ri\left(\Fcal_k^{\ast}\right)$, for all~$k\in\left\{1,\ldots,K\right\}$.

Assume that~$x^{l}\rightarrow{}x^{\ast}$.~For~$l$ large enough, for all~$k\in\left\{1,\ldots,K\right\}$ and all~$\alpha_k>0$, there exists~$\epsilon_k>0$ such that
\begin{align}
x_k^{l}\in\Bcal\left(x^{\ast}_k,\epsilon_k\right)&\cap\ri\left(\Fcal_k^{\ast}\right)\implies{}P_{\Omega_k}\left(x_k^{l}-t_k\nabla{}m_k\left(x_k^l\right)\right)\in\ri\left(\Fcal_k^{\ast}\right)\enspace,\nonumber
\end{align}
for all~$t_k\in\left]0,\alpha_k\right]$.
\end{lem}
\begin{proof}
Similarly to the proof of Lemma~$7.1$ in\ \citep{burke1990}, the idea is to show that there exists a neighbourhood of~$x^{\ast}_k$ such that if~$x_k^{l}$ lies in this neighbourhood, then 
\begin{align}
\nonumber
x_k^{l}-\alpha_k\nabla{}m_k\left(x_k^l\right)\in\ri\left(\Fcal_k^{\ast}+\Ncal\left(\Fcal_k^{\ast}\right)\right)\enspace.
\end{align}
Lemma\ \ref{lem:ac_detec} then follows by using the properties of the projection operator onto a closed convex set and Theorem~$2.3$ in\ \citep{burke1990}.

For simplicity, we prove the above relation for~$k=2$.~It can be trivially extended to all indices~$k$ in~$\left\{3,\ldots,K\right\}$.~Let~$\alpha_2>0$ and~$l\geq1$. 
\begin{align}
\nonumber
x_2^l-\alpha_2\nabla{}m_2\left(x_2^l\right)&=x_2^l-\alpha_2g_2\left(x^l\right)-\alpha_2E_2B\left(x^l\right)E_1^{\Trans}\left(z_1^l-x_1^l\right)\enspace,
\end{align}
where the matrix~$E_k$ is defined in\ \eqref{eq:def_E_k}.~As~$x^{\ast}$ is non-degenerate,
\begin{align}
\nonumber
x^{\ast}-\alpha_2g\left(x^{\ast}\right)\in\ri\left(\Fcal^{\ast}\right)+\ri\left(\Ncal\left(\Fcal^{\ast}\right)\right)\enspace.
\end{align}
However, as the sets~$\left\{\Fcal_k^{\ast}\right\}_{k=1}^K$ are convex, one has\ \citep{rock2009}
\begin{align}
\nonumber
\ri\left(\Fcal^{\ast}\right)=\ri\left(\Fcal_1^{\ast}\right)\times\ldots\ri\left(\Fcal_K^{\ast}\right)~\text{and}~\Ncal\left(\Fcal^{\ast}\right)=\Ncal\left(\Fcal_1^{\ast}\right)\times\ldots\times\Ncal\left(\Fcal_K^{\ast}\right)\enspace.
\end{align}
Hence,
\begin{align}
\nonumber
x_2^{\ast}-\alpha_2g_2\left(x^{\ast}\right)\in\ri\left(\Fcal_2^{\ast}\right)+\ri\left(\Ncal\left(\Fcal_2^{\ast}\right)\right)=\inter\left(\Fcal_2^{\ast}+\Ncal\left(\Fcal_2^{\ast}\right)\right)\enspace,  
\end{align}
by Theorem~$2.3$ in\ \citep{burke1990}.~By continuity of the objective gradient~$g$, there exists~$\delta_2>0$ such that
\begin{align}
\nonumber
\left\|x^l-x^{\ast}\right\|_2<\delta_2\implies{}x_2^l-\alpha_2g_2\left(x^l\right)\in\inter\left(\Fcal_2^{\ast}+\Ncal\left(\Fcal_2^{\ast}\right)\right)\enspace.
\end{align}
However, as shown beforehand (Lemma\ \ref{lem:lim_inf}),
\begin{align}
\lim_{l\rightarrow+\infty}\left\|z_1^l-x_1^l\right\|_2=0\enspace.\nonumber
\end{align}
Moreover,~$E_2B\left(x^l\right)E_1^{\Trans}$ is bounded above (Ass.\ \ref{ass:bnd_hss_mod}), subsequently for~$l$ large enough,
\begin{align}
\nonumber
x_2^l-\alpha_2\nabla{}m_2\left(x_2^l\right)\in\inter\left(\Fcal_2^{\ast}+\Ncal\left(\Fcal_2^{\ast}\right)\right)\subseteq\ri\left(\Fcal_2^{\ast}+\Ncal\left(\Fcal_2^{\ast}\right)\right)\enspace,
\end{align}
by Theorem~$2.3$ in\ \citep{burke1990}.~Then, Lemma\ \ref{lem:ac_detec} follows by properly choosing the radii~$\epsilon_k$ so that~$\sum_{k=1}^K\epsilon_k^2=\left(\min\left\{\delta_k\right\}_{k=1}^K\right)^2$.
\end{proof}
We have just shown that, for a large enough iteration count~$l$, if the primal iterate~$x^l$ is sufficiently close to the critical point~$x^{\ast}$ and on the same face~$\Fcal^{\ast}$, then the set of active constraints at the Cauchy point~$z^l$ is the same as the set of active constraints at~$x^{\ast}$. 
\begin{theo} 
\label{th:ac_detec}
If Assumptions\ \ref{ass:bnd_hss_mod},\ \ref{ass:bnd_blw},\ \ref{ass:smooth} and\ \ref{ass:reg_lim_pt} are fulfilled, then the following holds 
\begin{align} 
\lim_{l\rightarrow+\infty}\left\|\nabla_{\Omega}L\left(x^l\right)\right\|_2=0\enspace.\nonumber
\end{align}
Moreover, there exists~$l_0$ such that for all~$l\geq{}l_0$,
\begin{align}
\Acal_\Omega\left(x^l\right)=\Acal_\Omega\left(x^\ast\right)\enspace.\nonumber
\end{align}
\end{theo}
\begin{proof}
\looseness-1The reasoning of the proof of Theorem $7.2$ in\ \citep{burke1990} can be applied using Lemma\ \ref{lem:ac_detec} and line\ \ref{line:ref_end} in Algorithm\ \ref{algo:dis_tr}.~The first step is to show that the Cauchy point $z$ identifies the optimal active set after a finite number of iterations.~This is guaranteed by Theorem $2.2$ in\ \citep{burke1990}, since $\nabla_{\Omega}L\left(z^l\right)\rightarrow0$ by Theorem\ \ref{th:cauch_cv}, and the sequence $\left\{x^l\right\}$ converges to a non-degenerate critical point by Theorem\ \ref{th:full_cv}.~Lemma\ \ref{lem:ac_detec} is used to show that if $x^l$ is close enough to $x^\ast$, then the Cauchy point $z^l$ remains in the relative interior of the same face, and thus the active constraints do not change after some point.
\end{proof}
\looseness-1Theorem\ \ref{th:ac_detec} shows that the optimal active set is identified after a finite number of iterations, which corresponds to the behaviour of the gradient projection in standard trust region methods.~This fact is crucial for the local convergence analysis of the sequence~$\left\{x^l\right\}$, as fast local convergence rate cannot be obtained if the dynamics of the active constraints does not settle down.
\subsection{Local Convergence Rate}
\label{subsec:cv_rate}
\looseness-1In this paragraph, we show that the local convergence rate of the sequence~$\left\{x^l\right\}$ generated by\ \textsc{trap} is almost Q-superlinear, in the case where a Newton model is approximately minimised at every trust region iteration, that is 
\begin{align}
\nonumber
B=\nabla^2L\enspace,
\end{align}
in model\ \eqref{eq:model}.~Similarly to\ \eqref{eq:reg_g_hss}, one can define
\begin{align}
\label{eq:def_reg_hss}
H_\sigma:=H+\displaystyle\frac{\sigma}{2}I\enspace.
\end{align}
To establish fast local convergence, a key step is to prove that the trust region radius is ultimately bounded away from zero.~It turns out that the regularisation of the trust region problem\ \eqref{eq:prox_ref} plays an important role in this proof.~As shown in the next Lemma\ \ref{lem:tr_not_zero}, after a large enough number of iterations, the trust region radius does not interfere with the iterates and an inexact~Newton step is always taken at the refinement stage (Line\ \ref{line:ref_beg} to\ \ref{line:ref_end}), implying fast local convergence depending on the level of accuracy in the computation of the~Newton direction.~However,~Theorem~$7.4$ in\ \citep{burke1990} cannot be applied here, since due to the alternating gradient projections, the model decrease at the~Cauchy point cannot be expressed in terms of the projected gradient on the active face at the critical point. 
\begin{lem}
\label{lem:tr_not_zero}
If Assumptions\ \ref{ass:bnd_hss_mod},\ \ref{ass:bnd_blw},\ \ref{ass:smooth} and\ \ref{ass:reg_lim_pt} are fulfilled, then there exists an index~$l_1\geq1$ and~$\Delta^{\ast}>0$ such that for all~$l\geq{}l_1$,~$\Delta^l\geq\Delta^{\ast}$. 
\end{lem} 
\begin{proof}
The idea is to show that the ratio~$\rho$ converges to one, which implies that all iterations are ultimately successful, and subsequently, by the mechanism of Algorithm\ \ref{algo:dis_tr}, the trust region radius is bounded away from zero asymptotically.~For all~$l\geq1$,
\begin{align}
\label{eq:tr_ratio_1}
\left|\rho^l-1\right|&=\displaystyle\frac{\left|L\left(y^l\right)-L\left(x^l\right)-\left<g\left(x^l\right),y^l-x^l\right>-\displaystyle\frac{1}{2}\left<y^l-x^l,H\left(x^l\right)\left(y^l-x^l\right)\right>\right|}{m\left(x^l\right)-m\left(y^l\right)}\enspace.
\end{align}
However, 
\begin{align}
m^l\left(x^l\right)-m^l\left(y^l\right)&=m^l\left(x^l\right)-m^l\left(z^l\right)+m^l\left(z^l\right)-m^l\left(y^l\right)\nonumber\\
&\geq\displaystyle\frac{\underline{\eta}}{2}\left\|z^l-x^l\right\|_2^2+\displaystyle\frac{\underline{\sigma}}{2}\left\|y^l-z^l\right\|_2^2\nonumber\\
&\geq\displaystyle\frac{\min\left\{\underline{\eta},\underline{\sigma}\right\}}{2}\left(\left\|z^l-x^l\right\|_2^2+\left\|y^l-z^l\right\|_2^2\right)\nonumber\\
&\geq\displaystyle\frac{\min\left\{\underline{\eta},\underline{\sigma}\right\}}{2}\max\left\{\left\|z^l-x^l\right\|_2^2,\left\|y^l-z^l\right\|_2^2\right\}\enspace,\nonumber
\end{align}
and 
\begin{align}
\left\|p^l\right\|_2&\leq\left\|y^l-z^l\right\|_2+\left\|z^l-x^l\right\|_2\nonumber\\
&\leq2\max\left\{\left\|y^l-z^l\right\|_2,\left\|z^l-x^l\right\|_2\right\}\enspace.\nonumber
\end{align}
Hence, 
\begin{align}
m^l\left(x^l\right)-m^l\left(y^l\right)\geq\displaystyle\frac{\min\left\{\underline{\eta},\underline{\sigma}\right\}}{8}\left\|p^l\right\|_2^2\enspace.\nonumber
\end{align}
Moreover, using the mean-value theorem, one obtains that the numerator in\ \eqref{eq:tr_ratio_1} is smaller than 
\begin{align}
\displaystyle\frac{1}{2}\psi^l\left\|p^l\right\|_2^2\enspace,\nonumber
\end{align}
where 
\begin{align}
\label{eq:def_psi_l}
\psi^l:=\suprem_{\tau\in\left[0,1\right]}\left\|H\left(x^l+\tau{}p^l\right)-H\left(x^l\right)\right\|_2\enspace.
\end{align}
Subsequently, we have
\begin{align}
\left|\rho^l-1\right|\leq\displaystyle\frac{4}{\min\left\{\underline{\eta},\underline{\sigma}\right\}}\psi^l\enspace,\nonumber
\end{align}
and the result follows by showing that~$p^l$ converges to zero.~Take~$l\geq{}l_0$, where~$l_0$ is as in Theorem\ \ref{th:ac_detec}.~Thus, $p^l\in\Ncal\left(\Fcal^{\ast}\right)^{\perp}$.~However, from the model decrease, one obtains  
\begin{align}
\displaystyle\frac{1}{2}\left<p^l,H\left(x^l\right)p^l\right>\leq\left<-g\left(x^l\right),p^l\right>\enspace.\nonumber
\end{align}
From Theorem\ \ref{th:full_cv}, the sequence~$\left\{x^l\right\}$ converges to~$x^{\ast}$, which satisfies the strong second-order optimality condition\ \ref{ass:reg_lim_pt}.~Hence, by continuity of the hessian~$\nabla^2L$ and the fact that~$\Acal_\Omega\left(x^l\right)=\Acal_\Omega\left(x^{\ast}\right)$, one can claim that there exists~$l_1\geq{}l_0$ such that for all~$l\geq{}l_1$, for all~$v\in\Ncal_\Omega\left(x^l\right)^{\perp}=\Ncal\left(\Fcal^{\ast}\right)^{\perp}$,
\begin{align}
\left<v,H\left(x^l\right)v\right>\geq\kappa\left\|v\right\|_2^2\enspace.\nonumber
\end{align} 
Thus, by~Moreau's decomposition, it follows that
\begin{align}
\displaystyle\frac{\kappa}{2}\left\|p^l\right\|_2^2&\leq\left<P_{\Tcal_{\Omega}\left(x^l\right)}\left(-g\left(x^l\right)\right)+P_{\Ncal_\Omega\left(x^l\right)}\left(-g\left(x^l\right)\right),p^l\right>\nonumber\\
 &\leq\left\|P_{\Tcal_{\Omega}\left(x^l\right)}\left(-g\left(x^l\right)\right)\right\|_2\left\|p^l\right\|_2\enspace,\nonumber
\end{align}
since~$p^l\in\Ncal\left(\Fcal^{\ast}\right)^{\perp}$.~Finally,~$p^l$ converges to zero, as a consequence of Lemma\ \ref{lem:coor_opt_upbnd} and the fact that~$\left\|z^l-x^l\right\|_2$ converges to~$0$, by Lemma\ \ref{lem:coor_opt_upbnd} and the fact that the step-sizes~$\alpha_k$ are upper bounded for~$k\in\left\{1,\ldots,K\right\}$. 
\end{proof}
The refinement step in\ \textsc{trap} actually consists of a truncated~Newton method, in which the~Newton direction is generated by an iterative procedure, namely the  distributed sCG described in Algorithm\ \ref{algo:pcg}.~The Newton iterations terminate when the residual $\hat{s}$ is below a tolerance that depends on the norm of the projected gradient at the current iteration.~In Algorithm\ \ref{algo:pcg}, the stopping condition is set so that at every iteration $l\geq1$, there exists $\xi^l\in\left]0,1\right[$ satisfying
\begin{align}
\label{eq:inex_new_1}
\left\|Z^l\left(Z^l\right)^{\Trans}\left(g_{\sigma^l}\left(x^l\right)+H_{\sigma^l}\left(x^l\right)p^l\right)\right\|_2\leq\xi^l\left\|Z^l\left(Z^l\right)^{\Trans}g\left(x^l\right)\right\|_2\enspace.
\end{align} 
\looseness-1The local convergence rate of the sequence~$\left\{x^l\right\}$ generated by\ \textsc{trap} is controlled by the sequences~$\left\{\xi^l\right\}$ and~$\left\{\sigma^l\right\}$, as shown in the following Theorem.
\begin{theo}[Local linear convergence]
Assume that the direction $p$ yielded by Algorithm\ \ref{algo:pcg} satisfies\ \eqref{eq:inex_new_1} if $\left\|p\right\|_\infty\leq\gamma^{\ast}\Delta$ and $\Acal_\Omega\left(x\right)=\Acal_\Omega\left(x+p\right)$, given $\gamma^{\ast}\in\left]0,\gamma_2\right[$.~Under Assumptions\ \ref{ass:bnd_hss_mod},\ \ref{ass:bnd_blw},\ \ref{ass:smooth} and\ \ref{ass:reg_lim_pt}, for a small enough $\bar{\sigma}$, the sequence $\left\{x^l\right\}$ generated by\ \textsc{trap} converges Q-linearly to $x^{\ast}$ if~$\xi^{\ast}<1$ is small enough, where
\begin{align}
\nonumber
\xi^{\ast}:=\limsup_{l\rightarrow+\infty}\xi^l\enspace.
\end{align}    
If $\xi^{\ast}=0$, the Q-linear convergence ratio can be made arbitrarily small by properly choosing $\bar{\sigma}$, resulting in almost Q-superlinear convergence.
\end{theo}  
\begin{proof}
Throughout the proof, we assume that $l$ is large enough so that the active-set is $\Acal_{\Omega}\left(x^{\ast}\right)$ and that $p^l$ satisfies condition\ \eqref{eq:inex_new_1}.~This is ensured by Lemma\ \ref{lem:tr_not_zero} and Theorem\ \ref{th:ac_detec}, as the sequence $\left\{p^l\right\}$ converges to zero.~Thus, we can write $Z^l=Z^{\ast}$.~The orthogonal projection onto the subspace $\Ncal\left(\Fcal^{\ast}\right)^{\perp}$ is represented by the matrix $Z^{\ast}\left(Z^{\ast}\right)^{\Trans}$.~A first-order development yields a positive sequence $\left\{\delta^l\right\}$ converging to zero such that 
\begin{align}
\nonumber
\left\|Z^{\ast}\left(Z^{\ast}\right)^{\Trans}g\left(x^{l+1}\right)\right\|_2&\leq\left\|Z^{\ast}\left(Z^{\ast}\right)^{\Trans}\left(g\left(x^l\right)+H\left(x^l\right)p^l\right)\right\|_2+\delta^l\left\|p^l\right\|_2\nonumber\\
&\leq\displaystyle\frac{2\delta^l}{\kappa}\left\|Z^{\ast}\left(Z^{\ast}\right)^{\Trans}g\left(x^l\right)\right\|_2+\left\|Z^{\ast}\left(Z^{\ast}\right)^{\Trans}\left(g_{\sigma^l}\left(x^l\right)+H_{\sigma^l}\left(x^l\right)p^l\right)\right\|_2\nonumber\\
&~~~~~~~~~~~~~~~~~~~~~~~~+\bar{\sigma}\left\|Z^{\ast}\left(Z^{\ast}\right)^{\Trans}\left(\displaystyle\frac{p^l}{2}+z^l-x^l\right)\right\|_2\nonumber\\
&\leq\left(\displaystyle\frac{2\delta^l}{\kappa}+\xi^l\right)\left\|Z^{\ast}\left(Z^{\ast}\right)^{\Trans}g\left(x^l\right)\right\|_2\nonumber\\
&~~~~~~~~~~+\bar{\sigma}\left(\displaystyle\frac{1}{\kappa}+\displaystyle\frac{\left\|Z^{\ast}\left(Z^{\ast}\right)^{\Trans}\left(z^l-x^l\right)\right\|_2}{\left\|Z^{\ast}\left(Z^{\ast}\right)^{\Trans}g\left(x^l\right)\right\|_2}\right)\left\|Z^{\ast}\left(Z^{\ast}\right)^{\Trans}g\left(x^l\right)\right\|_2\enspace.\nonumber
\end{align}
\looseness-1where the second inequality follows from the last inequality in Lemma\ \ref{lem:tr_not_zero}, and the definition of $g_{\sigma}$ in Eq.\ \eqref{eq:reg_g_hss} and $H_{\sigma}$ in Eq.\ \eqref{eq:def_reg_hss}.~However, from the computation of the Cauchy point described in paragraph\ \ref{subsec:stp_size} and Assumption~\ref{ass:bnd_hss_mod}, the term 
\begin{align}
\nonumber
\displaystyle\frac{\left\|Z^{\ast}\left(Z^{\ast}\right)^{\Trans}\left(z^l-x^l\right)\right\|_2}{\left\|Z^{\ast}\left(Z^{\ast}\right)^{\Trans}g\left(x^l\right)\right\|_2}
\end{align} 
is bounded by a constant $C>0$.~Hence,
\begin{align}
\nonumber
\displaystyle\frac{\left\|Z^{\ast}\left(Z^{\ast}\right)^{\Trans}g\left(x^{l+1}\right)\right\|_2}{\left\|Z^{\ast}\left(Z^{\ast}\right)^{\Trans}g\left(x^l\right)\right\|_2}\leq\displaystyle\frac{2\delta^l}{\kappa}+\xi^l+\bar{\sigma}\left(\displaystyle\frac{1}{\kappa}+C\right)\enspace.
\end{align}
Moreover, a first-order development provides us with a constant $\Upsilon>0$ such that
\begin{align}
\left\|Z^{\ast}\left(Z^{\ast}\right)^{\Trans}g\left(x^l\right)\right\|_2\leq\left(\hat{B}+\Upsilon\right)\left\|x^l-x^{\ast}\right\|_2\enspace.\nonumber
\end{align}
There also exists a positive sequence $\left\{\epsilon^l\right\}$ converging to zero such that
\begin{align}
\nonumber
\left\|Z^{\ast}\left(Z^{\ast}\right)^{\Trans}g\left(x^{l+1}\right)\right\|_2\geq\left\|Z^{\ast}\left(Z^{\ast}\right)^{\Trans}H\left(x^{\ast}\right)\left(x^{l+1}-x^{\ast}\right)\right\|_2-\epsilon^l\left\|x^{l+1}-x^{\ast}\right\|_2\enspace.\nonumber
\end{align}
However, since $x^{l+1}-x^{\ast}$ lies in $\Ncal\left(x^{\ast}\right)^{\perp}$, $Z^{\ast}\left(Z^{\ast}\right)^{\Trans}\left(x^{l+1}-x^l\right)=x^{l+1}-x^l$.~Thus, by Assumption\ \eqref{eq:ssoc}, 
\begin{align}
\nonumber
\left\|Z^{\ast}\left(Z^{\ast}\right)^{\Trans}\nabla{}L\left(x^{l+1}\right)\right\|_2\geq(\kappa-\epsilon^l)\left\|x^{l+1}-x^{\ast}\right\|_2\enspace,
\end{align}
which implies that, for $l$ large enough, there exists $\bar{\epsilon}\in\left]0,\kappa\right[$ such that
\begin{align}
\nonumber
\left\|Z^{\ast}\left(Z^{\ast}\right)^{\Trans}g\left(x^{l+1}\right)\right\|_2\geq(\kappa-\bar{\epsilon})\left\|x^{l+1}-x^{\ast}\right\|_2\enspace.
\end{align}
Finally, 
\begin{align}
\nonumber
\displaystyle\frac{\left\|x^{l+1}-x^{\ast}\right\|_2}{\left\|x^l-x^{\ast}\right\|_2}\leq\displaystyle\frac{\hat{B}+\Upsilon}{\kappa-\bar{\epsilon}}\left(\frac{2\delta^l}{\kappa}+\xi^l+\bar{\sigma}\left(\displaystyle\frac{1}{\kappa}+C\right)\right)\enspace,
\end{align}
which yields the result.
\end{proof}
\section{Numerical Examples}
\label{sec:ac_pow}
\looseness-1The optimal AC power flow constitutes a challenging class of nonconvex problems for benchmarking optimisation algorithms and software.~It has been used very recently in the testing of a novel adaptive augmented Lagrangian technique\ \citep{curtis2014}.~The power flow equations form a set of nonlinear coupling constraints over a network.~Some distributed optimisation strategies have already been explored for computing OPF solutions, either based on convex relaxations\ \citep{tse2012} or nonconvex heuristics\ \citep{baldick1997}.~As the convex relaxation may fail in a significant number of cases\ \citep{grothey2013}, it is also relevant to explore distributed strategies for solving the OPF in its general nonconvex formulation.~Naturally, all that we can hope for with this approach is a local minimum of the OPF problem.~Algorithm\ \ref{algo:dis_tr} is tested on the augmented Lagrangian subproblems obtained via a polar coordinates formulation of the OPF equations, as well as rectangular coordinates formulations.~Our\ \textsc{trap} algorithm is run as an inner solver inside a standard augmented Lagrangian loop\ \citep{bert1982} and in the more sophisticated\ \textsc{lancelot} dual loop\ \citep{conn1991}.~More precisely, if the~OPF problem is written in the following form
\begin{align}
\label{eq:opf_ex}
\minimise_{x}f\left(x\right)&\\
\text{s.t.}~g\left(x\right)=0&\nonumber\\
x\in\Xcal&\enspace,\nonumber
\end{align}
where~$\Xcal$ is a bound constraint set, an augmented Lagrangian loop consists in computing an approximate critical point of the auxiliary program\\[\baselineskip]
\begin{minipage}{\textwidth}
\begin{align}
\label{eq:opf_al}
\minimise_{x\in\Xcal}L_{\varrho}(x,\mu):=f(x)+\left(\mu+\displaystyle\frac{\varrho}{2}g(x)\right)^{\Trans}g(x)
\end{align}
\end{minipage}
\\[\baselineskip]\looseness-1with~$\mu$ a dual variable associated to the power flow constraints and $\varrho>0$ a penalty parameter, which are both updated after a finite sequence of primal iterations in\ \eqref{eq:opf_al}.~Using the standard first-order dual update formula, only local convergence of the dual sequence can be proven\ \citep{bert1982}.~On the contrary, in the\ \textsc{lancelot} outer loop, the dual variable~$\mu$ and the penalty parameter~$\varrho$ are updated according to the level of satisfaction of the power flow (equality) constraints, resulting in global convergence of the dual sequence\ \citep{conn1991}.~In order to test\ \textsc{trap}, we use it to compute approximate critical points of the subproblems\ \eqref{eq:opf_ex}, which are of the form\ \eqref{eq:bnd_nlp}.~The rationale behind choosing\ \textsc{lancelot} instead of a standard augmented Lagrangian method as the outer loop is that\ \textsc{lancelot} interrupts the inner iterations at an early stage, based on a~KKT tolerance that is updated at every dual iteration.~Hence, it does not allow one to really measure the absolute performance of\ \textsc{trap}, although it is likely more efficient than a standard augmented Lagrangian for computing a solution of the OPF program.~Thus, for all cases presented next, we provide the results of the combination of\ \textsc{trap} with a basic augmented Lagrangian and\ \textsc{lancelot}.~The augmented Lagrangian loop is utilised to show the performance of\ \textsc{trap} as a bound-constrained solver, whereas\ \textsc{lancelot} is expected to provide better overall performance.~All results are compared to the solution yielded by the nonlinear interior-point solver\ \textsc{ipopt}\ \citep{waech2006} with the sparse linear solver\ \textsc{ma27}.~Finally, it is important to stress that the results presented in this Section are obtained from a preliminary\ \textsc{matlab} implementation, which is designed to handle small-scale problems.~The design of a fully distributed software would involve substantial development and testing, and is thus beyond the scope of this paper.    

\subsection{AC Optimal Power Flow in Polar Coordinates}
\label{subsec:polar_opf}
We consider the AC-OPF problem in polar coordinates
\begin{align}
\label{eq:polar_opf}
&\minimise\sum_{g\in\Gcal}c_0^{g}+c_1^{g}p_g^{G}+c_2\left(p_g^G\right)^2\\
&\text{s.t.}\nonumber\\
&\sum_{g\in\Gcal_b}p_g^G=\sum_{d\in\Dcal_b}P_d^D+\sum_{b'\in\Bcal_b}p_{bb'}^L+G_b^Bv_b^2\nonumber\\
&\sum_{g\in\Gcal_b}q_g^G=\sum_{d\in\Dcal_b}Q_d^D+\sum_{b'\in\Bcal_b}q_{bb'}^L-B_b^Bv_b^2\nonumber\\
&p_{bb'}^L=G_{bb}v_b^2+\left(G_{bb'}\cos{(\theta_b-\theta_{b'})}+B_{bb'}\sin{(\theta_b-\theta_{b'})}\right)v_bv_{b'}\nonumber\\
&q_{bb'}^L=-B_{bb}v_b^2+\left(G_{bb'}\sin{(\theta_b-\theta_{b'})}-B_{bb'}\cos{(\theta_b-\theta_{b'})}\right)v_bv_{b'}\nonumber\\
&\left(p_{bb'}^L\right)^2+\left(q_{bb'}\right)^2+s_{bb'}=\left(S_{bb'}^M\right)^2\nonumber\\
&v_b^L\leq{}v_b\leq{}v_b^U\nonumber\\
&p^L\leq{}p_g^G\leq{}p^U\nonumber\\
&q^L\leq{}q_g^G\leq{}q^U\nonumber\\
&s_{bb'}\geq0\enspace,\nonumber
\end{align}
\looseness-1which corresponds to the minimisation of the overall generation cost, subject to power balance constraints at every bus~$b$ and power flow constraints on every line~$bb'$ of the network, where~$\Gcal$ denotes the set of generators and~$\Gcal_b$ is the set of generating units connected to bus~$b$.~The variables~$p^G_g$ and~$q^G_g$ are the active and reactive power output at generator~$g$.~The set of loads connected to bus~$b$ is denoted by~$\Dcal_b$.~The parameters~$P^D_d$ and~$Q^D_d$ are the demand active and reactive power at load unit~$d$.~The letter~$\Bcal_b$ represents the set of buses connected to bus~$b$.~Variables~$p_{bb'}^L$ and~$q_{bb'}^L$ are the active and reactive power flow through line~$bb'$.~Variables~$v_b$ and~$\theta_b$ denote the voltage magnitude and voltage angle at bus~$b$.~Constants~$v_b^L$,~$v_b^U$ are lower and upper bounds on the voltage magnitude at bus~$b$.~Constants~$p^L$,~$p^U$,~$q^L$ and~$q^U$ are lower and upper bounds on the active and reactive power generation.~It is worth noting that a slack variable~$s_{bb'}$ has been added at every line~$bb'$ in order to turn the usual inequality constraint on the power flow through line~$bb'$ into an equality constraint.~The derivation of the optimal power flow problem in polar form can be found in\ \citep{zhu2009}.

\looseness-1As a simple numerical test example for\ \textsc{trap}, we consider a particular instance of~NLP\ \eqref{eq:polar_opf} on the~$9$-bus transmission network shown in~Fig.\ \ref{fig:nine_bus}.~As in\ \eqref{eq:opf_al}, the augmented Lagrangian subproblem is obtained by relaxing the equality constraints associated with buses and lines in\ \eqref{eq:polar_opf}.~The bound constraints, which can be easily dealt with via projection, remain unchanged.~One should notice that~NLP\ \eqref{eq:polar_opf} has partially separable constraints and objective, so that\ \textsc{lancelot} could efficiently deal with it, yet in a purely centralised manner.~In some sense, running\ \textsc{trap} in a\ \textsc{lancelot} outer loop can be seen as a first step towards a distributed implementation of\ \textsc{lancelot} for solving the~AC-OPF.~It is worth noting that the dual updates only require exchange of information between neighbouring nodes and lines.~However, each\ \textsc{lancelot} dual update requires a central communication, as the norm of the power flow constraints need to be compared with a running tolerance\ \citep{conn1991}.   
\begin{figure}[h!]
\begin{center}
\includegraphics[scale=0.5]{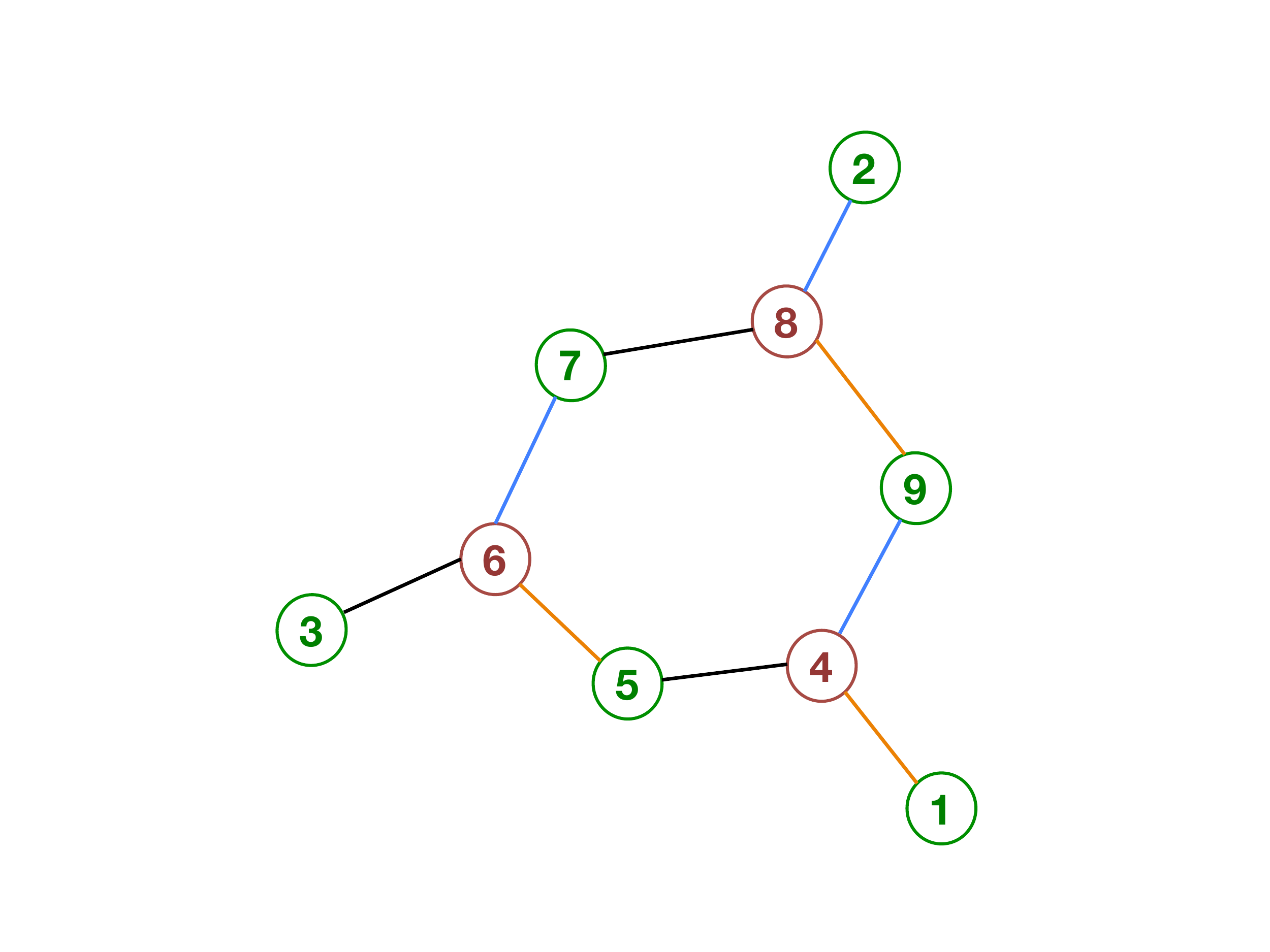}
\end{center}
\caption{\label{fig:nine_bus}The~$9$-bus transmission network from\ \url{http://www.maths.ed.ac.uk/optenergy/LocalOpt/}.}
\end{figure}
\looseness-1For the~$9$-bus example in~Fig.\ \ref{fig:nine_bus}, the~Cauchy search of\ \textsc{trap} on the augmented Lagrangian subproblem\ \eqref{eq:opf_al} can be carried out in five parallel steps.~This can be observed by introducing local variables for every bus~$b\in\left\{1,\ldots,9\right\}$,  
\begin{align}
x_b:=\left(v_b,\theta_b\right)^{\Trans}\enspace,\nonumber
\end{align}
and for every line 
\begin{align}
bb'\in\Big\{\left\{1,4\right\},\left\{4,5\right\},\left\{4,9\right\},\left\{8,9\right\},\left\{2,8\right\},\left\{7,8\right\},\left\{6,7\right\},\left\{3,6\right\},\left\{5,6\right\}\Big\}\enspace,\nonumber
\end{align}
with the line variable~$y_{bb'}$ being defined as
\begin{align}
y_{bb'}:=\left(p_{bb'},q_{bb'},s_{bb'}\right)^{\Trans}\enspace.\nonumber
\end{align}
\looseness-1The line variables~$y_{bb'}$ can be first updated in three parallel steps, which corresponds to
\begin{align}
\left\{y_{\left\{2,8\right\}},y_{\left\{6,7\right\}},y_{\left\{4,9\right\}}\right\},~\left\{y_{\left\{7,8\right\}},y_{\left\{3,6\right\}},y_{\left\{4,5\right\}}\right\},~\left\{y_{\left\{8,9\right\}},y_{\left\{5,6\right\}},y_{\left\{1,4\right\}}\right\}\enspace.\nonumber
\end{align}
Then, the subset 
\begin{align}
\left\{x_1,x_2,x_3,x_5,x_7,x_9\right\}\nonumber
\end{align} 
can be updated, followed by the subset 
\begin{align}
\left\{x_4,x_6,x_8\right\}\enspace.\nonumber
\end{align}
\looseness-1As a result, backtracking iterations can be run in parallel at the nodes associated with each line and bus.~If a standard trust region~Newton method would be applied, the projected search would have to be computed on the same central node without a bound on the number iterations.~Thus, the activity detection phase of\ \textsc{trap} allows one to reduce the number of global communications involved in the whole procedure.~The results obtained via a basic augmented Lagrangian loop and a\ \textsc{lancelot} outer loop are presented in Tables\ \ref{tab:res_9_opf_al} and\ \ref{tab:res_9_opf_lan} below.~The data is taken from the archive\ \url{http://www.maths.ed.ac.uk/optenergy/LocalOpt/}.~In all Tables of this Section, the first column corresponds to the index of the dual iteration, the second column to the number of iterations in the main loop of\ \textsc{trap} at the current outer step, the third column to the total number of~sCG iterations at the current outer step, the fourth column to the level of~KKT satisfaction obtained at each outer iteration, and the fifth column is the two-norm of the power flow equality constraints at a given dual iteration.
\begin{table}[h!]
\begin{center}
\begin{tabular}{|c|c|c|c|c|}
\hline
Outer iter.  &  \# inner it.  & \# cum. sCG &   Inner KKT   &  PF eq. constr. \\
 count  &   &    &     &   \\
\hline
\hline             
       $1$    &       $79$     &  $388$  &      $2.01\cdot10^{-7}$      &  $0.530$\\
       $2$     &      $2$       &  $40$    &      $2.71\cdot10^{-10}$    &   $0.530$\\
       $3$     &      $300$   &  $2215$  &    $2.39\cdot10^{-2}$      &  $0.292$\\
       $4$     &      $101$   &  $2190$  &    $6.50\cdot10^{-4}$      &  $6.56\cdot10^{-3}$ \\
       $5$     &      $123$   &  $2873$  &    $2.10\cdot10^{-3}$      &  $5.02\cdot10^{-6}$  \\
       $6$     &      $56$     &  $1194$  &    $4.14\cdot10^{-2}$      &  $1.11\cdot10^{-10}$  \\
\hline
\end{tabular}
\captionof{table}{\label{tab:res_9_opf_al}Results for the~$9$-bus~AC-OPF (Fig.\ \ref{fig:nine_bus}) using a standard augmented Lagrangian outer loop and\ \textsc{trap} as primal solver.~Note that the cumulative number of~CG iterations is relatively high, since the refinement stage was not preconditioned.}
\end{center}
\end{table}
\begin{table}[h!]
\begin{center}
\begin{tabular}{|c|c|c|c|c|}
\hline
Outer iter.  &  \# inner it.  &  \# cum. sCG  &  Inner KKT   & PF eq. constr. \\
 count  &   &     &     &   \\
\hline
\hline             
       $1$    &       $37$     &  $257$    &  $7.29\cdot10^{-2}$   &  $0.530$\\
       $2$     &      $5$       &   $25$     &  $1.01\cdot10^{-2}$    &   $0.530$\\
       $3$     &      $6$       &   $71$     &  $3.23\cdot10^{-5}$   &  $0.530$\\
       $4$     &      $100$   &   $1330$   &  $8.30\cdot10^{-3}$  &  $4.33\cdot10^{-2}$ \\
       $5$     &      $100$   &   $1239$  &  $1.80\cdot10^{-3}$   &  $2.53\cdot10^{-3}$  \\
       $6$     &      $100$   &   $2269$   &  $4.33\cdot10^{-2}$  &  $2.69\cdot10^{-5}$  \\
       $7$     &      $64$     &   $1541$   &  $3.2\cdot10^{-3}$   &  $1.64\cdot10^{-8}$  \\
\hline
\end{tabular}
\captionof{table}{\label{tab:res_9_opf_lan}Results for the~$9$-bus~AC-OPF~(Fig.\ \ref{fig:nine_bus}) using a\ \textsc{lancelot} outer loop and\ \textsc{trap} as primal solver.~Note that the cumulative number of~CG iterations is relatively high, since no preconditioner was applied in the refinement step.}
\end{center}
\end{table}
\looseness-1To obtain the results presented in Tables\ \ref{tab:res_9_opf_al} and\ \ref{tab:res_9_opf_lan}, the regularisation parameter~$\sigma$ in the refinement stage\ \ref{algo:pcg} is set to~$1\cdot10^{-10}$.~For~Table\ \ref{tab:res_9_opf_al}, the maximum number of iterations in the inner loop (\textsc{trap}) is fixed to~$300$ and the stopping tolerance on the level of satisfaction of the~KKT conditions to~$1\cdot10^{-5}$.~For~Table\ \ref{tab:res_9_opf_lan} (\textsc{lancelot}), the maximum number of inner iterations is set to~$100$ for the same stopping tolerance on the~KKT conditions.~In~Algorithm\ \ref{algo:pcg}, a block-diagonal preconditioner is applied.~It is worth noting that the distributed implementation of~Algorithm\ \ref{algo:pcg} is not affected by such a change.~To obtain the results of Table\ \ref{tab:res_9_opf_al}, the initial penalty parameter~$\varrho$ is set to~$10$ and is multiplied by~$30$ at each outer iteration.~In the\ \textsc{lancelot} loop, it is multiplied by~$100$.~In the end, an objective value of~$2733.55$ up to feasibility~$1.64\cdot10^{-8}$ of the power flow constraints is obtained, whereas the interior-point solver\ \textsc{ipopt}, provided with the same primal-dual initial guess, yields an objective value of~$2733.5$ up to feasibility~$2.23\cdot10^{-11}$.~From~Table\ \ref{tab:res_9_opf_al}, one can observe that a very tight~KKT satisfaction can be obtained with\ \textsc{trap}.~From the figures of~Tables\ \ref{tab:res_9_opf_al} and\ \ref{tab:res_9_opf_lan}, one can extrapolate that\ \textsc{lancelot} would perform better in terms of computational time~($6732$~sCG iterations in total) than a basic augmented Lagrangian outer loop~($8900$~sCG iterations in total), yet with a worse satisfaction of the power flow constraints~($1.64\cdot10^{-8}$ against~$1.11\cdot10^{-10}$).~Finally, one should mention that over a set of hundred random initial guesses,\ \textsc{trap} was able to find a solution satisfying the power flow constraints up to~$1\cdot10^{-7}$ in all cases, whereas\ \textsc{ipopt} failed in approximately half of the test cases, yielding a point of local infeasibility.

\subsection{AC Optimal Power Flow on Distribution Networks}
\label{subsec:dis_opf}
\looseness-1Algorithm\ \ref{algo:dis_tr} is then applied to solve two~AC-OPF problems in rectangular coordinates on distribution networks.~Both~$47$-bus and~$56$-bus networks are taken from\ \citep{topcu2014}.~Our results are compared against the nonlinear interior-point solver\ \textsc{ipopt}\ \citep{waech2006}, which is not amenable to a fully distributed implementation, and the~SOCP relaxation proposed by\ \citep{topcu2014}, which may be distributed (as convex) but fails in some cases, as shown next.~It is worth noting that any distribution network is a tree, so a minimum colouring scheme consists of two colours, resulting in~$4$ parallel steps for the activity detection in\ \textsc{trap}.
\subsubsection{On the~$56$-bus~AC-OPF}
\label{subsubsec:case56}
\looseness-1On the~$56$-bus~AC-OPF, an objective value of~$233.9$ is obtained with feasibility~$8.00\cdot10^{-7}$, whereas the nonlinear solver\ \textsc{ipopt} yields an objective value of~$233.9$ with feasibility~$5.19\cdot10^{-7}$ for the same initial primal-dual guess.

\looseness-1In order to increase the efficiency of\ \textsc{trap}, following a standard recipe, we build a block-diagonal preconditioner from the hessian of the augmented Lagrangian by extracting block-diagonal elements corresponding to buses and lines.~Thus, constructing and using the preconditioner can be done in parallel and does not affect the distributed nature of\ \textsc{trap}.~In Fig.\ \ref{fig:kkt_56}, the satisfaction of the~KKT conditions for the bound constrained problem\ \eqref{eq:opf_al} is plotted for a preconditioned refinement phase and non-preconditioned one.
\begin{figure}[h!]
\begin{center}
\includegraphics[width=0.9\textwidth]{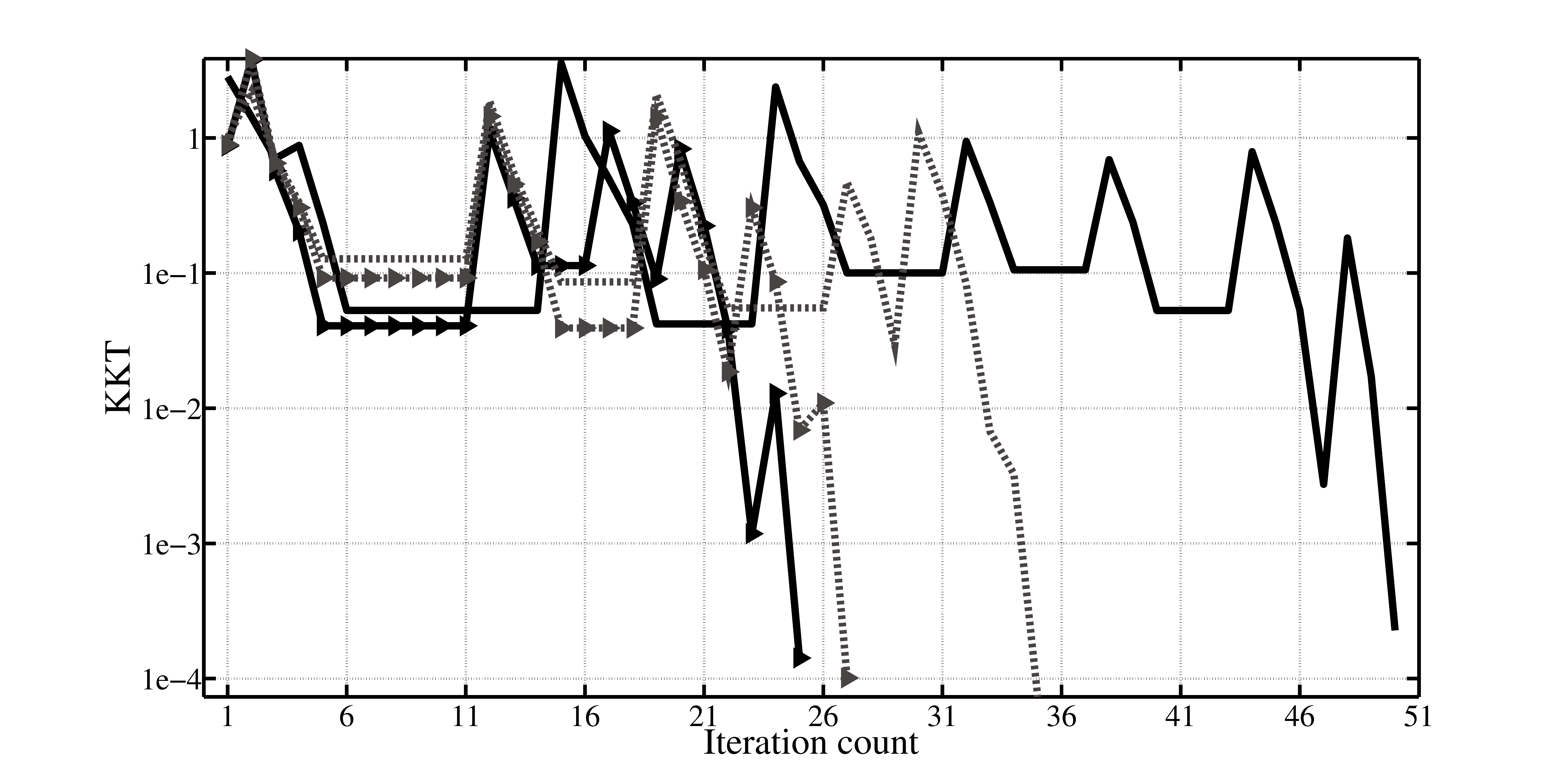}
\end{center}
\caption{\label{fig:kkt_56}KKT satisfaction vs iteration count in the fourth\ \textsc{lancelot} subproblem formed on the~AC-OPF with~$56$ buses.~When using a centralised projected search as activity detector (dotted grey) and\ \textsc{trap} (full black).~Curves obtained with a preconditioned sCG are highlighted with triangle markers.}
\end{figure}
\looseness-1One can conclude from Fig.\ \ref{fig:kkt_56} that preconditioning the refinement phase does not only affect the number of iterations of the sCG Algorithm\ \ref{algo:pcg} (Fig.\ \ref{fig:cg_56}), but also the performance of the main loop of\ \textsc{trap}.~From a distributed perspective, it is very appealing, for it leads to a strong decrease in the overall number of global communications.~Finally, from Fig.\ \ref{fig:kkt_56}, it appears that\ \textsc{trap} and a centralised trust region method (with centralised projected search) are equivalent in terms of convergence speed. 
\begin{figure}[h!]
\begin{center}
\includegraphics[width=0.9\textwidth]{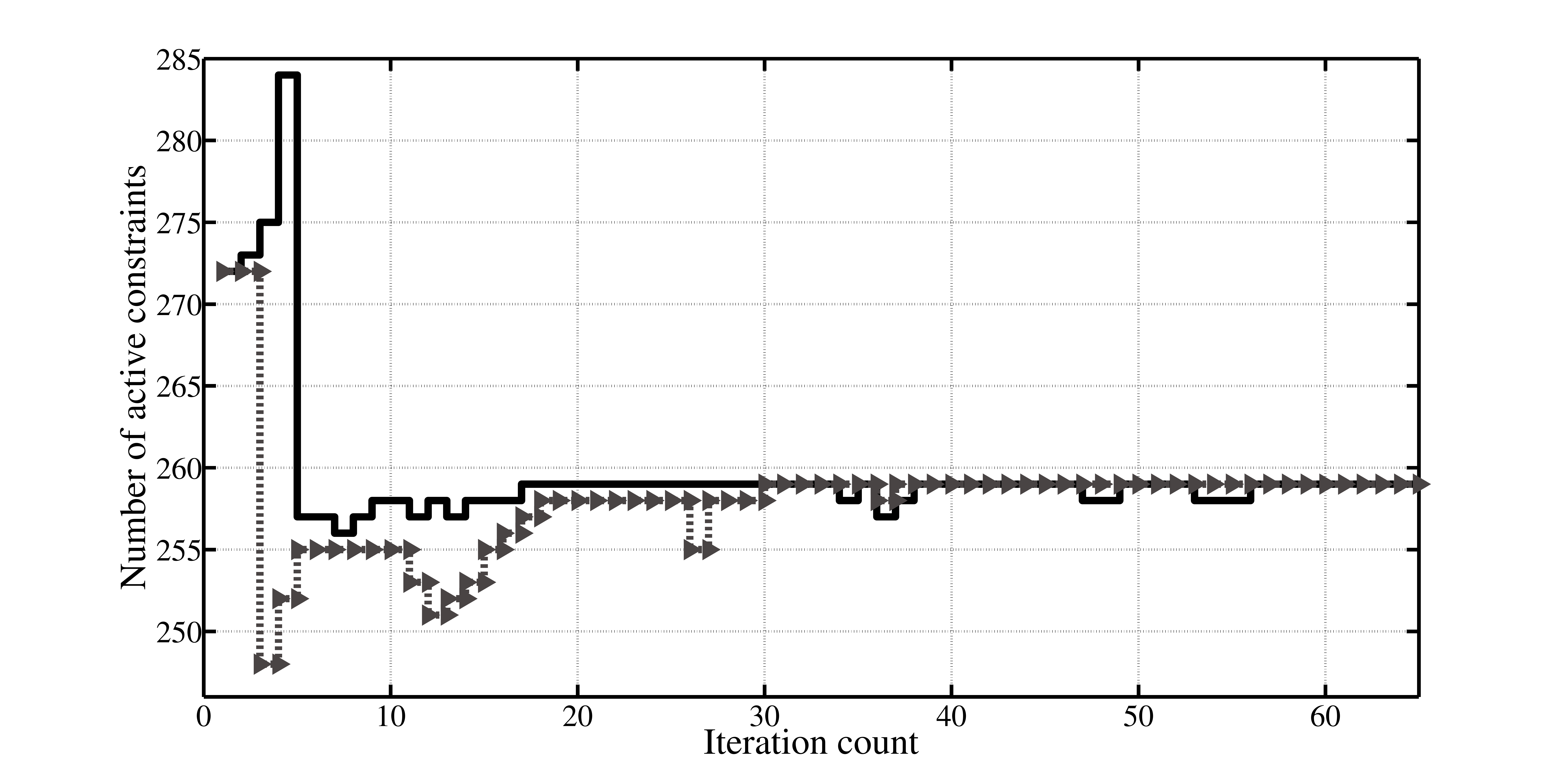}
\end{center}
\caption{\label{fig:as_56}Active-set history in the first\ \textsc{lancelot} iteration for the~$56$-bus~AC-OPF.~Activity detection in\ \textsc{trap}:\ \textsc{trap} (full black), centralised projected search (dashed grey with triangles).}
\end{figure}
\looseness-1From~Fig.\ \ref{fig:as_56},\ \textsc{trap} proves very efficient at identifying the optimal active set in a few iterations (more than~$10$ constraints enter the active-set in the first four iterations and about~$20$ constraints are dropped in the following two iterations), which is a proof of concept for the analysis of~Section\ \ref{sec:cv_ana}.~Alternating gradient projections appear to be as efficient as a projected search for identifying an optimal active-set, although the iterates travel on different faces, as shown in Fig.\ \ref{fig:as_56}.~In Fig.\ \ref{fig:nec_56}, the power flow constraints are evaluated after a run of\ \textsc{trap} on program\ \eqref{eq:opf_al}.~The dual variables and penalty coefficient are updated at each outer iteration.~Overall, the coupling of\ \textsc{trap} with the augmented Lagrangian appears to be successful and provides similar performance to the coupling with a centralised trust region algorithm.
\begin{figure}[h!]
\begin{center}
\includegraphics[width=0.9\textwidth]{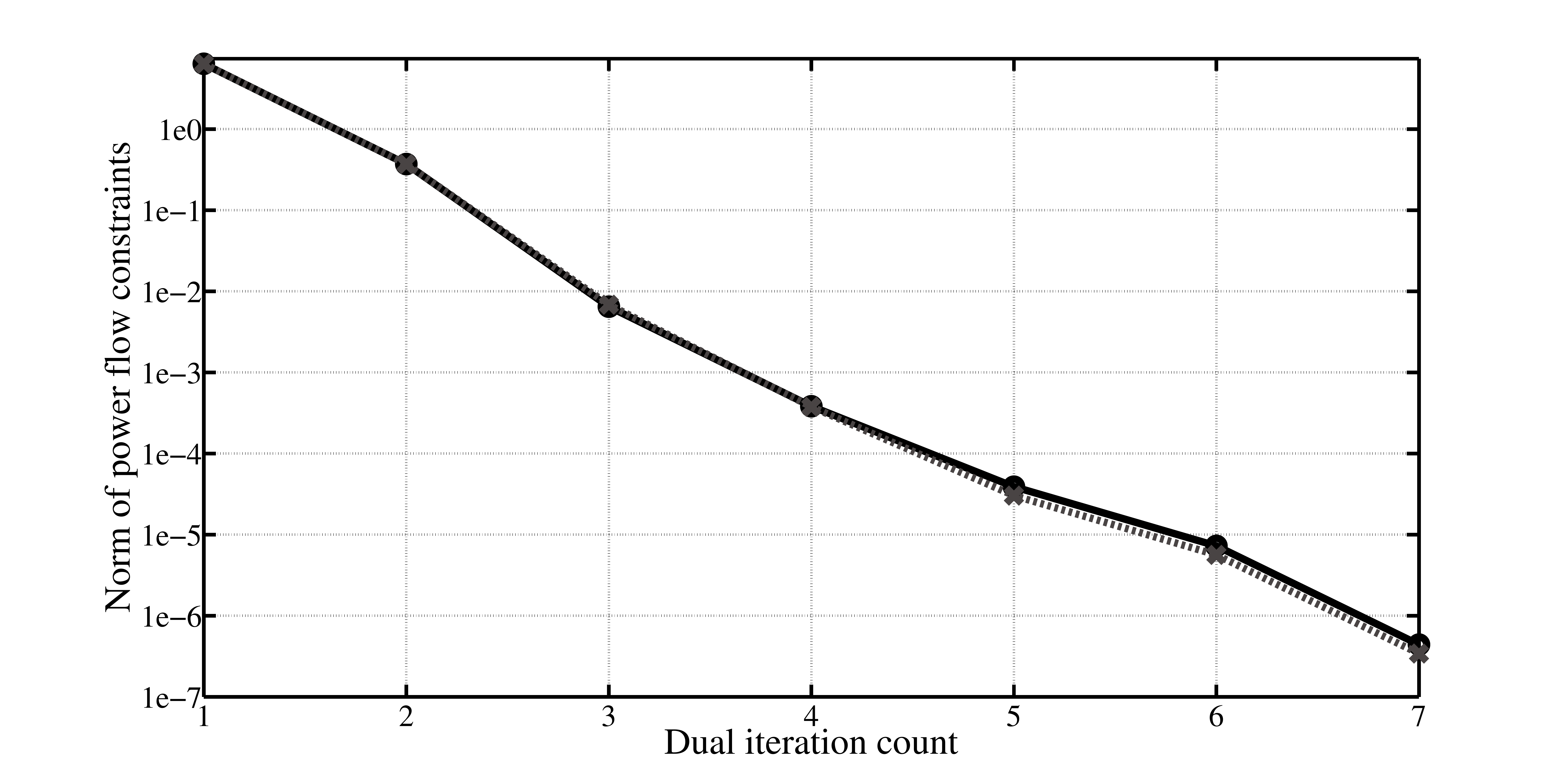}
\end{center}
\caption{\label{fig:nec_56}Norm of power flow constraints on the~$56$-bus network against dual iterations of a\ \textsc{lancelot} outer loop with\ \textsc{trap} as primal solver.~Inner solver:\ \textsc{trap} (full black), centralised trust region method (dashed grey with cross markers).}
\end{figure}
\begin{figure}[h!]
\begin{center}
\includegraphics[width=0.9\textwidth]{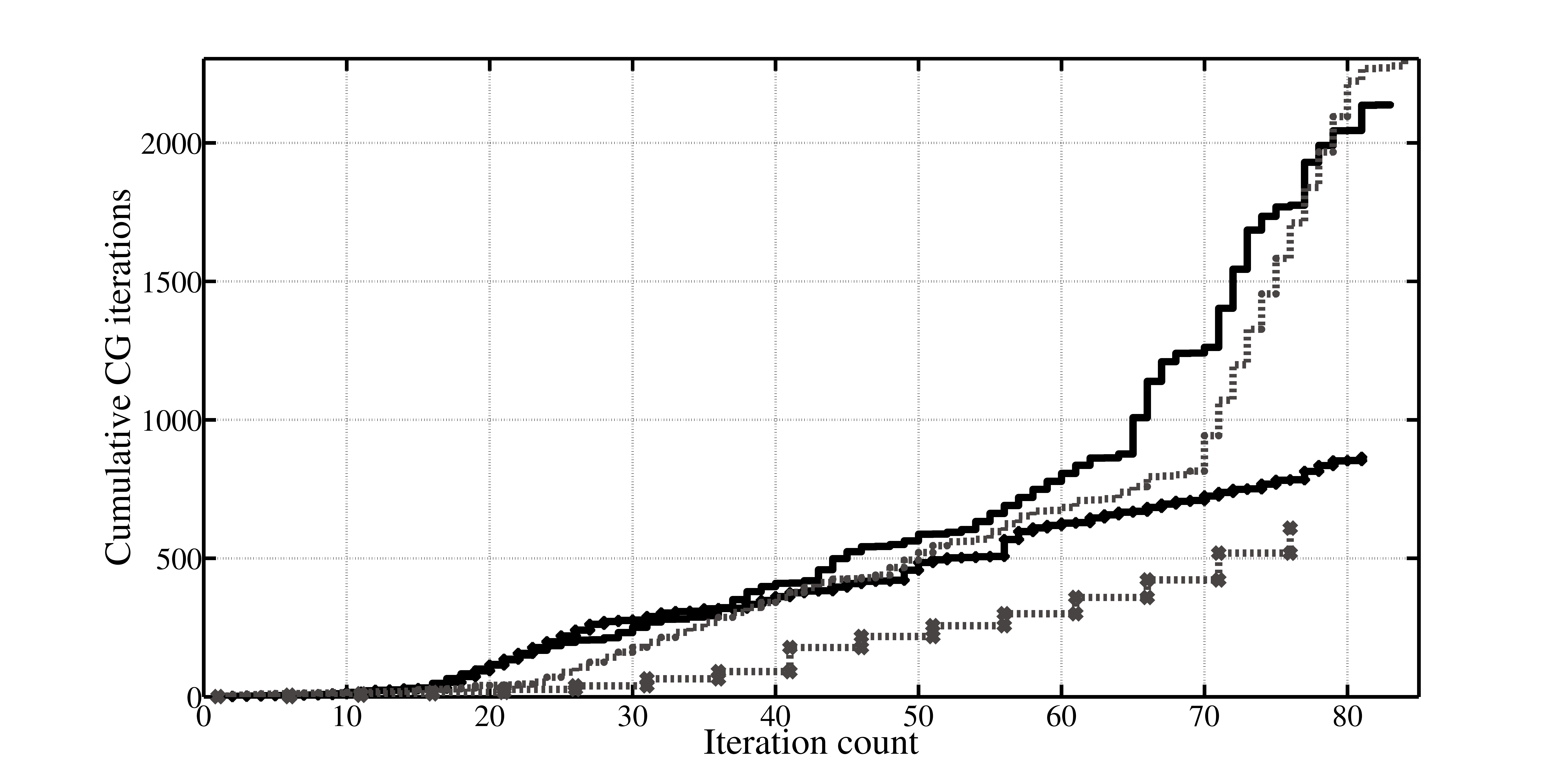}
\end{center}
\caption{\label{fig:cg_56}Cumulative~sCG iterations vs iteration count in the first\ \textsc{lancelot} subproblem formed on the~AC-OPF with~$56$ buses.~Results obtained with\ \textsc{trap} as inner solver (full black), with a centralised trust region method (dashed grey).~Results obtained with a preconditioned refinement stage are highlighted with cross markers.}
\end{figure}
\begin{table}[h!]
\begin{center}
\begin{tabular}{|c|c|c|c|c|}
\hline
Outer iter.  &  \# inner it.  &  \# cum. sCG  &  Inner KKT   & PF eq. constr. \\
 count  &   &     &     &   \\
\hline 
\hline 
       $1$      &     $122$   & $1382$    &  $8.45\cdot10^{-9}$    &  $6.68$\\
       $2$      &     $189$   & $4486$   &  $6.71\cdot10^{-9}$     &  $1.49\cdot10^{-1}$\\
       $3$      &     $139$   & $11865$   &  $9.87\cdot10^{-8}$     &  $8.79\cdot10^{-4}$\\
       $4$      &     $49$     & $3958$    &  $6.75\cdot10^{-6}$  &  $7.92\cdot10^{-6}$\\
       $5$      &     $9$       & $936$  &  $5.45\cdot10^{-7}$  &  $4.58\cdot10^{-9}$ \\    
\hline
\end{tabular}
\captionof{table}{\label{tab:res_56_opf_al}Results for the~$56$-bus~AC-OPF of\ \citep{topcu2014} using a (local) augmented Lagrangian outer loop with\ \textsc{trap} as primal solver.}
\end{center}
\end{table}
\begin{table}[h!]
\begin{center}
\begin{tabular}{|c|c|c|c|c|}
\hline
Outer iter.  &  \# inner it.  &  \# cum. sCG  &  Inner KKT   & PF eq. constr. \\
 count  &   &     &     &   \\
\hline 
\hline 
       $1$      &     $100$   & $924$    &  $9.74\cdot10^{-2}$    &  $6.42$\\
       $2$      &     $133$    &  $3587$   &  $2.40\cdot10^{-3}$     &  $3.60\cdot10^{-1}$\\
       $3$      &    $54$      &  $4531$   &  $1.03\cdot10^{-4}$     &  $4.00\cdot10^{-3}$\\
       $4$      &     $10$      & $858$    &  $4.20\cdot10^{-6}$  &  $1.02\cdot10^{-3}$\\
       $5$      &     $42$     &   $3288$  &  $4.37\cdot10^{-6}$  &  $2.32\cdot10^{-4}$ \\
       $6$      &     $13$     &   $916$  &  $1.82\cdot10^{-5}$    &  $4.35\cdot10^{-5}$ \\
       $7$      &     $40$      &  $6878$   &  $3.70\cdot10^{-7}$   &   $8.16\cdot10^{-6}$ \\
       $8$      &     $6$       &   $420$   &  $4.64\cdot10^{-6}$   &  $4.97\cdot10^{-7}$  \\      
\hline
\end{tabular}
\captionof{table}{\label{tab:res_56_opf_lan}Results for the~$56$-bus~AC-OPF of\ \citep{topcu2014} using a\ \textsc{lancelot} outer loop with\ \textsc{trap} as primal solver.}
\end{center}
\end{table}

\looseness-1Tables\ \ref{tab:res_56_opf_al} and\ \ref{tab:res_56_opf_lan} are obtained with an initial penalty coefficient~$\rho=10$ and a multiplicative coefficient of~$20$. 
 
\subsubsection{On the $47$-bus~AC-OPF}
\label{subsubsec:case47}
\looseness-1On the~$47$-bus~AC-OPF, a generating unit was plugged at node~$12$ (bottom of the tree) and the load at the substation was decreased to~$3$ pu.~On this modified problem, the~SOCP relaxation provides a solution, which does not satisfy the nonlinear equality constraints.~An objective value of~$502.3$ is obtained with feasibility~$2.57\cdot10^{-7}$ for both the~AL loop (Tab.\ \ref{tab:res_47_opf_al}) and the\ \textsc{lancelot} loop (Tab.\ \ref{tab:res_47_opf_lan}).
\begin{table}[h!]
\begin{center}
\begin{tabular}{|c|c|c|c|c|}
\hline
Outer iter.  &  \# inner it.  &  \# cum. sCG  &  Inner KKT   & PF eq. constr. \\
 count  &   &     &     &   \\
\hline   
\hline           
      $1$     &      $275$   & $3267$  &  $1.33\cdot10^{-7}$    &   $5.80$ \\
       $2$    &      $300$   & $7901$   &   $1.39\cdot10^{-1}$  &  $1.12\cdot10^{-1}$ \\
       $3$    &      $180$   & $18725$   &    $2.13\cdot10^{-6}$  &  $9.47\cdot10^{-5}$ \\
       $4$    &      $26$     & $3765$    &  $5.55\cdot10^{-8}$   &  $6.63\cdot10^{-9}$  \\
\hline
\end{tabular}
\captionof{table}{\label{tab:res_47_opf_al}Results for the~$47$-bus~AC-OPF of\ \citep{topcu2014} using an augmented Lagrangian outer loop with\ \textsc{trap} as primal solver.}
\end{center}
\end{table}
\begin{table}[h!]
\begin{center}
\begin{tabular}{|c|c|c|c|c|}
\hline
Outer iter.  &  \# inner it.  &  \# cum. sCG  &  Inner KKT   & PF eq. constr. \\
 count  &   &     &     &   \\
\hline  
\hline          
      $1$     &    $180$   &  $1147$   &  $8.64\cdot10^{-2}$    &   $5.35$ \\
       $2$    &    $300$   &  $7128$   &  $2.23$      &  $3.12\cdot10^{-1}$ \\
       $3$    &    $215$   &  $11304$  & $4.65\cdot10^{-5}$  &  $2.97\cdot10^{-3}$ \\
       $4$    &    $9$       &  $423$   &    $6.05\cdot10^{-5}$   &  $3.28\cdot10^{-5}$  \\
       $5$    &    $8$       &  $503$   &    $1.11\cdot10^{-8}$   &   $7.90\cdot10^{-7}$ \\
       $6$    &    $2$       &  $177$   &    $4.64\cdot10^{-6}$   &   $4.03\cdot10^{-8}$ \\
\hline
\end{tabular}
\captionof{table}{\label{tab:res_47_opf_lan}Results for the~$47$-bus~AC-OPF of\ \citep{topcu2014} using a\ \textsc{lancelot} outer loop with\ \textsc{trap} as primal solver.}
\end{center}
\end{table}
\looseness-1The~SOCP relaxation returns an objective value of~$265.75$, but physically impossible, as the power flow constraints are not satisfied.~The nonlinear solver\ \textsc{ipopt} yields an objective value of~$502.3$ with feasibility~$5.4\cdot10^{-8}$.
\section{Conclusions}
A novel trust region~Newton method, entitled\ \textsc{trap}, which is based on distributed activity detection, has been described and analysed.~In particular, as a result of a proximal regularisation of the trust region problem with respect to the~Cauchy point yielded by an alternating projected gradient sweep, global and fast local convergence to first-order critical points has been proven under standard regularity assumptions.~It has been argued further how the approach can be implemented in distributed platforms.~The proposed strategy has been successfully applied to solve various nonconvex~OPF problems, for which distributed algorithms are currently raising interest.~The performance of the novel activity detection mechanism compares favourably against the standard projected search.
\bibliographystyle{spbasic}	
\bibliography{biblio}	

\end{document}